\documentclass[10pt]{amsart}
\usepackage{amsmath,amssymb,amsthm,hyperref}
\usepackage{graphicx}
\usepackage{tikz}  
\newtheorem{theorem}{Theorem}
\newtheorem{corollary}[theorem]{Corollary}

\newtheorem{lemma}[theorem]{Lemma}
\newtheorem{proposition}[theorem]{Proposition}
\newtheorem{example}{Example}
{\theoremstyle{definition}
	\newtheorem{definition}{Definition}
	\newtheorem{remark}{Remark}

	}

\newcommand{\N}{\ensuremath{\mathbb N}} 
\newcommand{\R}{\ensuremath{\mathbb R}} 

\newcommand{\E}{\ensuremath{\mathbb E}}

\title[Asymptotic covariances for functionals of random fields]{Asymptotic covariances for functionals of weakly stationary random fields}

\author{Leonardo Maini}
\address{Leonardo Maini, Universit\'e du Luxembourg, 
	Unit\'e de Recherche en Math\'ematiques,
	Maison du Nombre,
	6 avenue de la Fonte,
	L-4364 Esch-sur-Alzette,
	Grand Duch\'e du Luxembourg}
\email{leonardo.maini@uni.lu}

\title{Asymptotic covariances for functionals of weakly stationary random fields}

\date{\today}
\begin{document}
	\maketitle

	\medskip

	\begin{abstract}
		Let $(A_x)_{x\in\R^d}$ be a measurable, weakly stationary random field, i.e. $\E[A_x]=\E[A_y]$, ${\rm Cov}(A_x,A_y)=K(x-y)$, $\forall x,y\in\R^d$, with covariance function $K:\R^d\rightarrow\R$. 
		
		Assuming only that the \textit{integral covariance function} $w_t:=\int_{\{|z|\le t\}}K(z)dz$ is regularly varying (which encompasses the classical assumptions found in the literature), we compute $$\lim_{t\rightarrow\infty}{\rm Cov}\left(\frac{\int_{tD}A_x dx}{t^{d/2}w_t^{1/2}}, \frac{\int_{tL}A_y dy}{t^{d/2}w_t^{1/2}}\right)$$ for $D,L\subseteq \mathbb{R}^d$ belonging to a certain class of compact sets. 
		
		As an application, we combine this result with existing limit theorems to obtain multi-dimensional limit theorems for non-linear functionals of stationary Gaussian fields, in particular proving new results for the Berry's random wave model. At the end of the paper, we also show how the problem for $A$ with a general continuous covariance function $K$ can be reduced to the same problem for a radial, continuous covariance function $K_{\text{iso}}$.
		
		The novel ideas of this work are mainly based on regularity conditions for (cross) covariograms of Euclidean sets and standard properties of regularly varying functions.
	\end{abstract}


\section{Introduction}
Fix a probability space $(\Omega,\mathcal{F},\mathbb{P})$, a dimension $d\ge1$ and let $A=(A_x)_{x\in\R^d}$ be a \textbf{measurable} random field on $(\Omega,\mathcal{F},\mathbb{P})$, that is,
	\begin{align*}
		A\colon\left(\Omega,\mathcal{F}\right)\times \left(\R^d, \mathcal{B}\left(\R^d\right)\right)\longrightarrow \left(\R, \mathcal{B}\left(\R\right)\right), \quad \quad\quad \quad(\omega,x) \longmapsto A_x(\omega),
	\end{align*}
is a measurable function, where $\mathcal{B}(\R^d)$ denotes the Borel sigma-algebra on $\R^d$. In addition, assume that $A$ is \textbf{weakly stationary}, i.e.
\begin{equation*}
	\label{2oc}
	\E\left[A_x\right]=\E[A_y], \quad\quad{\rm Cov}(A_x,A_y)=K(x-y),\quad \quad \forall x,y\in\R^d,
\end{equation*}
with \textbf{covariance function} $K:\R^d\rightarrow\R$. Note that under the above assumptions $K$ is measurable and bounded, hence locally integrable. Therefore, we can define the \textbf{integral covariance function}
$w_{\cdot}:\R_+\rightarrow\R$ as
\begin{equation}
	\label{w}
	w_t:=\int_{\{|z|\le t\}}K(z)dz, \quad \quad t>0,
\end{equation}
where $|\cdot|$ denotes the Euclidean norm on $\mathbb R^d$.

In this paper, we will focus on \textbf{functionals} of $A$ of the form
\begin{equation}
	\label{set indexed}
	\left(\int_{tD} A_x dx\right)_{D\in\mathcal{D}}, \quad \quad \text{as }t\rightarrow\infty,
\end{equation}
where $\mathcal{D}$ is a suitable class of compact sets in $\R^d$. The well-posedness of  \eqref{set indexed} under the above assumptions is ensured by Proposition \ref{wellposedness}.

The asymptotic behavior of \eqref{set indexed} has been extensively investigated in probability, and plays an important role in many applications. In Statistics, for instance, functionals of the form $\int_{tD}A_x dx$ are often involved in parameter estimation, and $tD$ plays the role of an increasing observation window, see e.g. \cite{DOV,MY}. 

Moreover, note that \eqref{set indexed} may represent non-linear functionals of stationary Gaussian fields, that will be discussed in more detail in Section \ref{nnfGF}. The latter corresponds to the choice $A_x=\varphi(B_x)$, where $B=(B_x)_{x\in\R^d}$ is a stationary Gaussian field and $\varphi:\R\rightarrow \R$ is a function which varies depending on the framework. Just to mention two examples (among many), $\varphi$ could be a polynomial in view of statistical applications, see e.g. \cite{RST,TV}, or an indicator function if one is interested in the geometry (and in particular the excursion volumes) of $B$, see e.g. \cite{LO,LRM,MN}. \\

The goal of this paper will be to give the minimal assumptions on $K$ and $\mathcal{D}$ in order to compute exactly  the \textbf{asymptotic covariances} of \eqref{set indexed}, i.e. the limit
\begin{equation}\label{problem}
	\lim_{t\rightarrow\infty}{\rm Cov}\left(\frac{\int_{tD}A_x dx}{r_t}, \frac{\int_{tL}A_y dy}{r_t}\right), \quad \quad D,L\in\mathcal{D},
\end{equation}
where $r_t\rightarrow\infty$ is chosen so that the limit \eqref{problem} exists finite (and not identically equal to $0$ for all $D,L\in\mathcal{D}$) for every $D,L\in\mathcal{D}$.
 The problem (\ref{problem}) often arises when studying the fluctuations of (\ref{set indexed}), in particular in the case  of non-linear functionals of stationary Gaussian fields (see Section \ref{nnfGF}). Indeed, the first step to prove limit theorems for \eqref{set indexed}, i.e.
 \[
 \frac{\int_{tD}(A_x-\E[A_x])dx}{r_t}\overset{{\rm d}}{\rightarrow}\,Z(D), \quad \quad \text{as $t\rightarrow\infty$},
 \]
  where $Z(D)$ is a limiting random variable, is usually to study the asymptotic variance of $\int_{tD}A_x dx$, which corresponds to \eqref{problem} when $D=L$. Moreover, to extend such results to a multi-dimensional setting, i.e.
  \[
  \left(\frac{\int_{tD_1}(A_x-\E[A_x])dx}{r_t},\dots,\frac{\int_{tD_n}(A_x-\E[A_x])dx}{r_t}\right)\overset{{\rm d}}{\rightarrow}\,(Z(D_1),\dots,Z(D_n)), \quad \quad \text{as $t\rightarrow\infty$},
  \]
    it is often necessary to compute \eqref{problem} for $D=D_i$ and $L=D_j$, for all $i,j=1,\dots,n$.\\
  
  A number of different methods have been used in the literature to compute \eqref{problem}, such as: 
 \begin{itemize}
 	\item Spectral representations and Fejer-type kernels or approximate identity for convolutions, see e.g. \cite{ALS}.
 	\item Spectral representations and Abelian-Tauberian theorems, see e.g. \cite{L}.
 	\item When $K$ is radial, the method of geometric probabilities, see e.g. \cite{LRM}.
 	\item A direct approach. By Fubini's theorem one has \footnote{Here the equality is ensured by Proposition \ref{wellposedness}.}
 \begin{equation}
 	\label{covfirsteq}
 		{\rm Cov}\left(\int_{tD}A_x dx, \int_{tL}A_y dy\right)=\int_{tD}\int_{tL}K(x-y)dxdy.
 \end{equation}
 \end{itemize}
 In this paper, we will focus on this latter direct approach, combining it with regularity conditions for cross covariograms developed in \cite{Galerne} and standard properties of regularly varying functions. This method will allow to compute \eqref{problem} under assumptions that encompass the classical ones found in the literature. 

\subsection{Classical assumptions, existing literature, motivating examples}\label{secmotex}
To highlight the main novelty introduced by this paper, let us look at the classical assumptions considered in the literature.

The \textbf{first classical assumption}, probably the most popular in the literature (see e.g. \cite{Ben Hariz,BM,CNN,KS,LO,NPP}), is
\begin{equation}
	\label{Kintegrable}
	K\in L^1(\R^d),
\end{equation}
see also Section \ref{sectionBM}. In this case, it is well known that for every $D,L\subseteq \R^d$ compact, choosing $r_t=t^{d/2}$, we have
\begin{equation}
	\label{Kintegrable ascov}
	\lim_{t\rightarrow\infty}{\rm Cov}\left(\frac{\int_{tD}A_x dx}{t^{d/2}}, \frac{\int_{tL}A_y dy}{t^{d/2}}\right)=\lim_{t\rightarrow\infty}\int_{\R^d} K(z)\left(\int_{D\cap L+\frac{z}{t}}dx\right) dz={\rm Vol}(D\cap L)\int_{\R^d}K(z)dz,
\end{equation}
where the latter follows by \eqref{covfirsteq}, a change of variable $z=x-y$ and dominated convergence.
Then, we have two possibilities:
\begin{itemize}
	\item If $\int_{\R^d}K(z)dz=0$, then  $r_t=t^{d/2}$ is not a correct choice for computing the asymptotic covariances, since the limit is identically $0$ for every $D,L$ compact (see also Example \ref{ex:bm0}). 
	\item If $\int_{\R^d}K(z)dz\neq 0$, then $r_t=t^{d/2}$ is a correct choice for computing the asymptotic covariances \eqref{problem}, as well as $r_t=t^{d/2} w_t^{1/2}$ (recall the definition \eqref{w} of $w$ and note that $K\in L^1(\R^d)$), and we have
	\begin{equation}
		\label{Kintegrable cov adapted}
			\lim_{t\rightarrow\infty}{\rm Cov}\left(\frac{\int_{tD}A_x dx}{t^{d/2}w_t^{1/2}}, \frac{\int_{tL}A_y dy}{t^{d/2}w_t^{1/2}}\right)={\rm Vol}(D\cap L).
	\end{equation}
	
\end{itemize}

The \textbf{second classical assumption}, which covers some of the cases where $K\notin L^1(\R^d)$, is $K$ radial and regularly varying with index $-\beta\in(-d,0)$ (see e.g. \cite{Anh,DM,LO,Rosenblatt,Taqqu}), namely
\begin{equation}
	\label{K rrv}
	K(z)=k(|z|)=\frac{\ell(|z|)}{|z|^{\beta}},\quad \quad \beta\in(0,d),
\end{equation}
for some $k:\R_+\rightarrow \R$, where $|\cdot|$ is the Euclidean norm in $\R^d$ and $\ell:\R_+\rightarrow\R$ is a slowly varying function, (i.e. $\ell$ is definitively positive and $\ell(rs)/\ell(r)\rightarrow 1$ for $s>0$, as $r\rightarrow\infty$, see the seminal book \cite{Bingham} or Section \ref{regularly varying} for more details). 
In this case, it is a standard fact that for every $D,L\subseteq \R^d$ compact, choosing $r_t=t^{d}k(t)^{1/2}=t^{d-\frac{\beta}{2}}\ell(t)^{1/2}$, we have
\begin{equation}
	\label{K rrv ascov}
	\lim_{t\rightarrow\infty}{\rm Cov}\left(\frac{\int_{tD}A_x dx}{t^{d}k(t)^{1/2}}, \frac{\int_{tL}A_y dy}{t^{d}k(t)^{1/2}}\right)=\lim_{t\rightarrow\infty}\int_{D}\int_L \frac{k(t|x-y|)}{k(t)} dx dy=\int_D \int_L |x-y|^{-\beta}dxdy,
\end{equation}
where the latter follows by \eqref{covfirsteq}, Theorem \ref{potter} and dominated convergence (see also Remark \ref{just}). This means that $r_t=t^{d}k(t)^{1/2}$ is a correct choice to compute the asymptotic covariances \eqref{K rrv ascov}. Moreover, by Theorem \ref{nonviceversa} we have \footnote{Here $\omega_0=2$ and $\omega_{d-1}$ is the surface area of the sphere $S^{d-1}$ in $\R^d$ when $d\ge2$.}
\[
w_t=\omega_{d-1}\int_0^t k(r)r^{d-1}dr\sim\frac{\omega_{d-1}}{(d-\beta)}t^d k(t)=\frac{\omega_{d-1}}{(d-\beta)}t^{d-\beta} \ell(t), \quad \quad \text{as $t\rightarrow\infty$},
\]
which implies that $r_t=t^{d/2}w_t^{1/2}$ is also a correct choice, and
\begin{equation}
		\label{K rrv ascov adapted}
		\lim_{t\rightarrow\infty}{\rm Cov}\left(\frac{\int_{tD}A_x dx}{t^{d/2}w_t^{1/2}}, \frac{\int_{tL}A_y dy}{t^{d/2}w_t^{1/2}}\right)=\frac{(d-\beta)}{\omega_{d-1}}\int_D \int_L |x-y|^{-\beta}dxdy.
\end{equation}

Summarizing, and  excluding the case $K\in L^1(\R^d)$, $\int_{\R^d}K(z)dz=0$, we have:
\begin{enumerate}
	\item Under both assumptions, $r_t=t^{d/2}w_t^{1/2}$ is a correct choice for computing the asymptotic covariances.\label{rem1}
	\item Under both assumptions $w_t$ is a regularly varying function, with index: $\alpha=0$ if $K\in L^1(\R^d)$, $\int_{\R^d}K(z)dz\neq0$; $\alpha=d-\beta\in(0,d)$ if $K$ is radial and regularly varying with index $-\beta\in(-d,0)$.\label{rem2}
\end{enumerate}
The intuition we can develop from (i)-(ii) is: if $w_t$ is regularly varying, then $r_t=t^{d/2}w_t^{1/2}$ is a correct choice for computing \eqref{problem}. This fact  will be proved in our main result, Theorem \ref{isoasymptotics}, together with an explicit expression for the asymptotic covariances (see \eqref{covgen}), which generalizes \eqref{Kintegrable cov adapted} and \eqref{K rrv ascov adapted} for $D,L$ in a suitable class $\mathcal{D}$ of compact domains (see \eqref{case1} and \eqref{case2}). \\

It is important to note that, excluding the case $K\in L^1(\R^d)$, $\int_{\R^d}K(z)dz=0$,  our assumption "$w_t$ regularly varying" (see \eqref{rv}) in Theorem \ref{isoasymptotics} is \textbf{strictly more general than the two forementioned classical assumptions}. This is the main novelty provided by the present paper.

In fact, in many cases (see e.g. \cite[Lemma 2.6]{GMT}, \cite[Example 5]{L}, or \cite[Section 6]{MN}) the covariance function $K$ is neither integrable nor radial and regularly varying, but the integrated covariance function $w$ is regularly varying. Consider, for instance, the situations where $K$ is not absolutely integrable, but $\lim_{t\rightarrow\infty}w_t=\int_{\R^d}K(z)dz \in(0,\infty)$; or $w_t\sim t^{\alpha}$, $\alpha>0$, but $K$ is not radial and regularly varying.
Some explicit examples will be given in Section \ref{sectionberry}.\\

The case $K\in L^1(\R^d)$, $\int_{\R^d}K(z)dz=0$ was not considered in the discussion above, since even if it falls under the first classical assumption, one is not able to compute the asymptotic covariances with \eqref{Kintegrable ascov}. However, if $w$ is regularly varying, Theorem \ref{isoasymptotics} allows to derive an expression for the asymptotic covariances, with the rate $r_t=t^{d/2}w_t^{1/2}$. An explicit example (see Example \ref{ex:bm0}) will be given in Section \ref{sectionBM}.\\
 
As mentioned above, the main novelty provided by Theorem \ref{isoasymptotics} is the computation of \eqref{problem} in many cases which do not fall under the two classical assumptions.

However, as we will better explain in Section \ref{nnfGF}, Theorem \ref{isoasymptotics} can also be combined with the Peccati-Tudor theorem \cite{PT} to prove \textbf{new multi-dimensional central limit theorems for non-linear functionals of Gaussian fields}.

An example of this fact, that stands out for its importance in quantum mechanics, is the \textbf{Berry's random wave model}, a Gaussian field whose functionals have been extensively studied in the literature (see Section \ref{sectionberry} and the references therein). Recently, it was proved in \cite{MN} that a large class of functionals of this field (and of many other fields satisfying a specific spectral condition) in the form (\ref{set indexed phi})  have Gaussian fluctuations. Nevertheless, the authors were not able to extend their result to a multi-dimensional central limit theorem (see also \cite[Section 1.4]{MN}). This paper also aims to partially fill this gap. Indeed, as we will better explain in Section \ref{sectionberry}, by means of Theorem \ref{isoasymptotics} we will prove new multi-dimensional central limit theorems for the Berry's random wave model (see in particular Example \ref{Berry4} and Example \ref{Berry2}), using the fact that $w$ is regularly varying (i.e. (\ref{rv}) holds) even if $K$ does not always satisfy the forementioned classical assumptions.

\subsection{Main result}
In order to state our main result, we need to introduce some notations, quantities and assumptions.

We denote by ${\rm "Vol"}$ the Lebesgue measure and by  ${\rm "Per"}$ the perimeter in generalized sense in $\R^d$, defined for $D\subseteq \R^d$ measurable as (see \cite{Galerne})
\[
{\rm Per}(D) =\sup\left\{\int_{D}{\rm div }\phi(x)\,dx :\phi\in C^1_c(\R^d,\R^d), \|{\phi}\|_\infty\le 1\right\},
\]
where $C^1_c(\R^d,\R^d)$ is the set of continuously differentiable functions with compact support. Note that, denoting by $\mathcal{H}^{d-1}(\partial D)$ the $(d-1)$-Hausdorff measure of the topological boundary $\partial D$, we have ${\rm Per}(D)\le \mathcal{H}^{d-1}(\partial D)$, which becomes an equality under additional assumptions on $D$ (see \cite{Galerne} for more details).

Furthermore, for $f:\R^d\rightarrow \R$ and $r>0$, $$\int_{S^{d-1}}f(r\theta)d\theta$$is the integral of $\theta\mapsto f(r\theta)$ (when defined) with respect to the uniform measure on $S^{d-1}$ and $\omega_{d-1}=\int_{S^{d-1}}d\theta$ is the surface measure of $S^{d-1}$. (For $d=1$ we use the formalism
\begin{equation}
	\label{not1}
	\int_{S^0}f(r\theta)d\theta=f(r)+f(-r), \quad\quad \omega_0=2.\text{)}
\end{equation}
For $D,L\in\mathcal{D}$, we consider the \textbf{cross covariogram} $g_{D,L}:\R^d\rightarrow\R$,
\begin{equation}
	\label{g}
	g_{D,L}(z)={\rm Vol}(D\cap (L+z)),
\end{equation}
which, by Proposition \ref{prop crosscov}, is such that all its directional derivatives $$\frac{d}{dl}g_{D,L}(l\theta), \quad\theta\in S^{d-1},$$ are well-defined in $L^1(\R_+)$ and bounded by ${\rm Per}(D)\wedge{\rm Per}(L)$.\\

Finally, we will always assume to be in one of the following cases:

\medskip

\noindent
{\it \textbf{Case 1.}} $K$ is \textbf{radial} and $\mathcal{D}$ is the class of all compact sets in $\R^d$ with \textbf{finite perimeter}, i.e. for some $k:\R_+\rightarrow\R$ we have
\begin{equation}
	\label{case1}
	K(z)=k(|z|), \quad \mathcal{D}=\{D\subseteq \R^d : D \text{ compact, }{\rm Per}(D)<\infty \}.
\end{equation}
If $d=1$, $K$ is always radial. 

\medskip

\noindent
{\it \textbf{Case 2.}} $\mathcal{D}=\mathcal{D}_{x_0}$ is the class of the \textbf{closed balls centered} at a fixed point  $x_0\in\R^d$, namely 
\begin{equation}
	\label{case2}
	\mathcal{D}=\left\{\left\{x\in\R^d : |x-x_0|\le r\right\}: r\in \R_+\right\}.
\end{equation}

\medskip
We are now ready to state our main result.
\begin{theorem}
	\label{isoasymptotics}
	Let $\mathcal{D}$ be a collection of compact sets in $\R^d$ and $A=(A_x)_{x\in\R^d}$ be a measurable, weakly stationary random field with covariance function $K:\R^d\rightarrow\R$. Furthermore, let $K$ and $\mathcal{D}$ satisfy the assumption (\ref{case1}) or the assumption (\ref{case2}), and assume that $w$ in (\ref{w}) is \textbf{regularly varying} with index $\alpha\in(-1,d]$, that is
	\begin{equation}
		\label{rv}
		w_t=\ell(t)t^{\alpha} \quad \forall t\in\R_+
	\end{equation}
	where $\ell:\R_+\rightarrow\R$ is  slowly varying (see Definition \ref{regularly varying} and Lemma \ref{lemma add meas lin}). 
	
	Then, for all  $D,L\in \mathcal{D}$, we have as $t\rightarrow\infty$ 
	\begin{equation}
		\label{covgen}
		{\rm Cov}\left(\frac{\int_{tD}A_x dx}{t^{d/2}w_t^{1/2}}, \frac{\int_{tL}A_x dx}{t^{d/2}w_t^{1/2}}\right)\to \frac{1}{\omega_{d-1}}\int_{S^{d-1}}d\theta\int_0^\infty dl\left(-\frac{d}{dl}g_{D,L}(l\theta)\right)l^{\alpha}.
	\end{equation}
\end{theorem}

\smallskip

\begin{remark}{\textbf{(Restriction of the parameter $\mathbf{\alpha}$)}} 
	Note that in Theorem \ref{isoasymptotics} we consider $\alpha\in(-1,d]$ instead of $\alpha\in\mathbb{R}$. This happens for different reasons:
	\begin{itemize}
		\item $\alpha>-1$ is assumed for technical reasons. For example, if $D=L=\{|z|\le1\}\subseteq \R^d$, we have for all $\theta\in S^{d-1}$ and some constant $c_d>0$
		\[
		g_{D,D}(l\theta)=c_d \int_{\frac{l}{2}}^1 \left(1-r^2\right)^{\frac{d-1}{2}} dr \,\, \mathbf{1}_{[0,2]}(l), \quad \quad\quad -\frac{d}{dl}g_{D,D}(l\theta)=\frac{c_d}{2}\left(1-\left(\frac{l}{2}\right)^2 \right)^{\frac{d-1}{2}}\mathbf{1}_{[0,2]}(l),
		\]
		implying  that the RHS of (\ref{covgen}) is not a finite integral if $\alpha\le -1$. Anyway, if (\ref{rv}) holds with $\alpha\le -1$, another scaling could yield a different limit.
		\smallskip
		\item Since $|K(z)|\le K(0)<\infty$ for every $z\in\mathbb R^d$, we have
		\[
		|w_t|\le\int_{\{|z|\le t\}}|K(z)|dz\le {\rm const}\,K(0)t^d,
		\]
		implying that $w_t$ can not be regularly varying with index $\alpha>d$.
		
	\end{itemize}
\end{remark}
\begin{remark}{\textbf{(Comparison to known covariance structures)}} 
	\label{comparison}Before going on, it is worth discussing analogies and differences with some known covariance structures, highlighting when (\ref{covgen}) can be seen as their generalization. First of all, note that when $d=1$, $D=[0,s]$ and $L=[0,r]$ with $r>s>0$, then for $l>0$ we have
	\begin{eqnarray}
		g_{D,L}(l)&=&(s-l)\mathbf{1}_{[0,s]}(l)\nonumber\\
		g_{D,L}(-l)&=&s\mathbf{1}_{[0,r-s]}(l)+(r-l)\mathbf{1}_{[r-s,r]}(l) \nonumber\\
		-\frac{d}{dl}(g_{D,L}(l)+g_{D,L}(-l))&=&\mathbf{1}_{[0,s]}(l)+\mathbf{1}_{[r-s,r]}(l) \nonumber
	\end{eqnarray}
	and the RHS of (\ref{covgen}) becomes
	\begin{equation}
		\label{covFBM}
		\frac{1}{2}\left(\int_0^s l^{\alpha}dl+\int_{r-s}^r l^{\alpha}dl \right)=\frac{1}{2(\alpha+1)}\left(s^{\alpha+1}+r^{\alpha+1}-|r-s|^{\alpha+1} \right).
	\end{equation}
	In \eqref{covFBM} we observe the covariance structure of a fractional Brownian motion $(B^H_s)_{s>0}$ with Hurst index $H=\frac{\alpha+1}{2}$. Note that this is also the covariance structure of other (non-Gaussian) stochastic processes, like the Hermite process (see e.g. \cite[Section 3]{tudor-book}).
	
	For $\alpha=0$, the RHS of (\ref{covgen}) has the special form
	\begin{equation}
		\label{cov0}
		{\rm Vol}(D\cap L),
	\end{equation}
	which is the covariance structure of a Gaussian noise (a set-indexed generalization of the Brownian motion, see e.g. \cite[Section 1.4.3]{Adler}). Note that if $D$ is fixed and $u,v>0$, \eqref{covgen} becomes
	\[
	\lim_{t\rightarrow\infty}{\rm Cov}\left(\frac{\int_{tD u^{1/d}}A_x dx}{t^{d/2}}, \frac{\int_{tD v^{1/d}}A_y dy}{t^{d/2}}\right)= {\rm Vol}(D)\,\min\{u,v\},
	\]
	which is, up to a scaling factor, the covariance structure of a Brownian motion (see also \eqref{bmmultispecialcase}). This rescaling is used in the literature to prove multidimensional central limit theorems (to a Brownian motion) for functionals of Gaussian fields, see e.g. \cite[Theorem 2.3.1]{IL} and the discussion in Section \ref{sectionBM}.
	
	For $\alpha>0$, the RHS of (\ref{covgen}) coincides with 
	\begin{equation}
		\label{cov>0}
		\frac{1}{\omega_{d-1}}\int_{S^{d-1}}d\theta\int_0^\infty dl\left(-\frac{d}{dl}g_{D,L}(l\theta)\right)l^{\alpha}=\frac{\alpha}{\omega_{d-1}}\int_{S^{d-1}}d\theta\int_0^\infty dlg_{D,L}(l\theta)l^{\alpha-1}=\frac{\alpha}{\omega_{d-1}}\int_{\R^d}g_{D,L}(z) |z|^{\alpha-d}dz,
	\end{equation} 
	where the first equality follows integrating by parts and the second passing from polar to standard coordinates. Note that
	\[
	\int_{\R^d}g_{D,L}(z) |z|^{\alpha-d}dz=\int_{\R^d}\left(\int_{\R^d}\mathbf{1}_D(x)\mathbf{1}_{L}(x-z)dx\right) |z|^{\alpha-d}dz=\int_D\int_L|x-y|^{\alpha-d}dxdy.
	\] 
	Thus, \eqref{cov>0} is exactly the RHS of \eqref{K rrv ascov adapted} with $\beta=d-\alpha$, which is perfectly consistent with the discussion in Section \ref{secmotex}.
	
	Finally, note that choosing $d\ge2$, $D=[0,\underline{s}]$ and $L=[0,\underline{r}]$, with $\underline{r},\underline{s}\in \R^d$, we have:
	\begin{itemize}
		\item If $\alpha=0$, then  (\ref{cov0}) is the covariance structure of a $d$-dimensional Brownian sheet.
		\item If $\alpha\neq0$, then the RHS of (\ref{covgen}) is not the covariance structure of a $d$-dimensional fractional Brownian sheet (see e.g. \cite{Wang} for the definition).
	\end{itemize}
\end{remark}

\subsection{Plan of the paper}
The rest of the paper is organized as follows. In Section \ref{prelim} we ensure the well-posedness of \eqref{set indexed}, prove \eqref{covfirsteq} and give some preliminaries about cross covariograms and regularly varying functions. In Section \ref{exact asymptotics} we prove Theorem \ref{isoasymptotics}. In Section \ref{nnfGF} we apply Theorem \ref{isoasymptotics} to prove multi-dimensional central limit theorems for non-linear functionals of stationary Gaussian fields, giving several examples. Finally, in Section \ref{sectionfinal} we show that $\eqref{problem}$ for $A$ with continuous covariance function $K$ can be reduced to the same problem for a radial, continuous covariance function $K_{\text{iso}}$, providing a class of non-Gaussian, non-stationary, weakly stationary random fields where this reduction principle can be applied.

\section{Preliminaries for the proof of Theorem \ref{isoasymptotics}}
\label{prelim}
The goal of this section is to ensure the well-posedness of \eqref{set indexed}, prove \eqref{covfirsteq} and give preliminary results on cross covariograms and regularly varying functions. 
\subsection{Well-posedness of \eqref{set indexed} and proof of \eqref{covfirsteq}}
The following proposition explains why \eqref{set indexed} is  well posed and \eqref{covfirsteq} holds.
\begin{proposition}
	\label{wellposedness}
	Fix $d\ge1$ and let $A=(A_x)_{x\in\R^d}$ be a measurable (in the sense that $(\omega,x)\mapsto A_x(\omega)$ is measurable), weakly stationary random field, with covariance function $K:\R^d\rightarrow \R$. Then:
	\begin{enumerate}
		\item We have
		\[
		\mathbb{P}\left(\int_{D}|A_x|dx<\infty, \quad \forall \text{ $D\subseteq \R^d$ compact}\right)=1
		\]which implies that \eqref{set indexed} is well defined.
		\item For every $D,L\subseteq \R^d$ compact we have
		\[
		{\rm Cov}\left(\int_{D}A_x dx,\int_{L}A_ydy\right)=\int_D \int_L K(x-y)dxdy,
		\]
		which implies \eqref{covfirsteq}.
	\end{enumerate}
\end{proposition}
\begin{proof}
	For simplicity, since $A$ is weakly stationary, let us use the notation $m:=\E[A_x]$, $\forall $ $x\in\R^d$.
	Note that the function $(\omega,x,y)\mapsto |A_x(\omega)-m||A_y(\omega)-m|$ is measurable because we are assuming that $(\omega,x)\mapsto A_x(\omega)$ is measurable. Moreover, we have $$|A_x-m||A_y-m|\le \frac{1}{2}\left(|A_x-m|^2+|A_y-m|^2\right),$$
	which implies 
	\begin{equation}
		\label{niceona}
		\E\left[\int_D\int_L\left|A_x-m\right||A_y-m|dxdy\right] \le \frac{1}{2}\left({\rm Var}(A_x)+{\rm Var}(A_y)\right){\rm Vol}(D){\rm Vol}(L)= K(0){\rm Vol}(D){\rm Vol}(L)<\infty.
	\end{equation}
	In particular, for $D\subseteq \R^d$ compact we have 
	\[
	\E\left[\int_{D}|A_x-m| dx\right]\le \E\left[\left(\int_{D}|A_x-m| dx\right)^2\right]^{1/2}=\E\left[\int_D\int_D\left|A_x-m\right||A_y-m|dxdy\right]^{1/2} <\infty,
	\] 
	implying that for almost every $\omega\in \Omega$ the function $x\mapsto A_x(\omega)$ is integrable on the compact $D$. To prove that for almost every $\omega\in\Omega$ the function $x\mapsto A_x(\omega)$ is integrable on every compact set of $\R^d$ (i.e. the almost sure local integrability stated in (i)), note that as $n\rightarrow\infty$ $$\Omega_n:=\left\{\omega\in\Omega: \int_{\{|x|\le n\}}|A_x(\omega)|dx<\infty \right\} \searrow \Omega_\infty:=\left\{\omega\in\Omega : \int_{D}|A_x(\omega)|dx<\infty, \forall D\subseteq \R^d \text{ compact}\right\} .$$
	Thus, since $\mathbb{P}(\Omega_n)=1$ for every $n\in\N$, we have $\mathbb{P}(\Omega_\infty)=1$ and (i) is proved.
	
	Since $\eqref{set indexed}$ is now well defined by (i) and \eqref{niceona} holds, by Fubini-Tonelli we obtain
	\begin{align*}
		{\rm Cov}\left(\int_{D}A_x dx,\int_{L}A_ydy\right)&=\E\left[\int_D\int_L(A_x-m)(A_y-m)dxdy\right]\\&=\int_D\int_L{\rm Cov}(A_x,A_y)dxdy=\int_D \int_L K(x-y)dxdy,
	\end{align*}
	for every $D,L$ compact domains, which concludes the proof of (ii).
\end{proof}

\subsection{Cross covariograms}
First of all, let us recall the definition (\ref{g}) of the cross covariogram $g_{D,L}$ of two compact sets $D,L$ in $\R^d$. In particular, when $D=L$, denote by $g_D$ the \textbf{covariogram}  of $D$
\begin{equation*}
	g_D(z):=g_{D,D}(z)={\rm Vol}(D\cap (D+z)),\quad \quad \quad z\in\R^d.
\end{equation*}
Moreover, for a bounded set $E\subseteq \R^d$, recall the definition of diameter
\begin{equation}
	\label{h}
	{\rm diam}(E):=\sup_{x,y\in E}|x-y|<\infty.
\end{equation}
Let us now list some properties related to cross covariograms, which are easily derived from the results in \cite{Galerne} and will be needed in the sequel.

\begin{proposition}\label{prop crosscov}
	Consider $D,L$ compact sets in $\R^d$. Then we have:
	\begin{enumerate}
		\item $g_{D,L}(z)=g_{L,D}(-z)$ and in particular $g_D$ is symmetric.
		\smallskip
		\item For every $z,h\in\R^d$
		\[
		|g_{D,L}(z+h)-g_{D,L}(z)|\le 2 \min\{g_D(0)-g_D(h),g_L(0)-g_L(h)\}.
		\]
		\item $g_{D,L}(z)=0$ if $|z|>{\rm diam}(D \cup L)$. In particular, since ${\rm diam}(D \cup L)<\infty$, $g_{D,L}$ has compact support.
		\smallskip
		\item ${\rm Per}(D)<\infty$ if and only if $g_D$ is Lipschitz continuous (with Lipschitz constant ${\rm Per}(D)/2$).
		\smallskip
		\item If ${\rm Per}(D)<\infty$ or ${\rm Per}(L)<\infty$, then $g_{D,L}$ is Lipschitz with Lipschitz constant ${\rm Per}(D)\wedge {\rm Per}(L)$. Moreover, all the directional derivatives of $g_{D,L}$ exist almost everywhere and are bounded by ${\rm Per}(D)\wedge {\rm Per}(L)$. In addition, for every $\theta\in S^{d-1}$ the function $r\rightarrow g_{D,L}(r\theta)$ is absolutely continuous and we have
		\begin{equation}
			\label{covbyparts}
			g_{D,L}(r\theta)=\int_r^{{\rm diam}(D \cup L)}\left(-\frac{d}{dl}g_{D,L}(l\theta)\right)dl, \quad \quad \quad\forall r>0.
		\end{equation}
	\end{enumerate} 
\end{proposition}
\begin{proof}
	(i) follows  from ${\rm Vol}(D\cap (L+z))={\rm Vol}((D-z)\cap L)$
	and (ii) from
	\begin{eqnarray*}
		&&|{\rm Vol}(D\cap (L+z+h))-{\rm Vol}(D\cap (L+z))|\le\int_{\R^d}|\mathbf{1}_{L+z}(x)-\mathbf{1}_{L+z+h}(x)|dx\\ &&=\int_{\R^d}|\mathbf{1}_{L+z}(x)-\mathbf{1}_{L+z+h}(x)|^2dx
		=\int_{\R^d}\mathbf{1}_{L+z}(x)+\mathbf{1}_{L+z+h}(x)-2\mathbf{1}_{L+z}(x)\mathbf{1}_{L+z+h}(x)\,dx\\&&=2(g_L(0)-g_L(h))
	\end{eqnarray*}
	and using (i).
	(iii) simply follows by definition, since $g_{D,L}(z)\neq0$ is possible only if $z\in D-L\subseteq (D\cup L)-(D\cup L)$, and (iv) is proved in \cite[Theorem 14]{Galerne}. Regarding (v), the fact that $g_{D,L}$ is Lipschitz with Lipschitz constant ${\rm Per}(D)\wedge {\rm Per}(L)$ easily follows from (iv) and (ii). As a consequence, for every fixed $\theta\in S^{d-1}$ the function
	\[
	l\mapsto g_{D,L}(l\theta)
	\]
	is Lipschitz, its derivative exists almost everywhere and is bounded by ${\rm Per}(D)\wedge {\rm Per}(L)$. Moreover, the fact that every Lipschitz function is absolutely continuous, together with (iii), implies (\ref{covbyparts}).\\
\end{proof}

\subsection{Regularly varying functions}
\label{regularly varying section}

\begin{definition}
	\label{regularly varying}
	A measurable function $h:\R_+\rightarrow\R$ is said \textbf{regularly varying} if $h$ is positive on $[a,\infty)$ for some $a>0$ and if  we have, for all $l>0$
	\begin{equation*}
		\label{regularly varying equation}
		\frac{h(tl)}{h(t)}\rightarrow g(l)\quad \quad \text{as } t\rightarrow\infty.
	\end{equation*}
	In particular, if $g\equiv 1$ then $h$ is said \textbf{slowly varying}. 
\end{definition}

Definition \ref{regularly varying} is one of the equivalent definitions used for regularly varying functions, as explained in Lemma \ref{lemma add meas lin} (see \cite[Section 1.4]{Bingham}).

\begin{lemma}
	\label{lemma add meas lin} 
	Consider a measurable function $h:\R_+\rightarrow\R$. Then the following statements are equivalent:
	\begin{enumerate}
		\item \label{1}$h$ is regularly varying.
		\item \label{2}$h$ is regularly varying with limit $g(l)=l^{\alpha}$, for some $\alpha\in\R$.
		\item \label{3}$h(l)=\ell(l)l^{\alpha}$ for some $\alpha\in\R$ and $\ell:\R_+\rightarrow\R$ slowly varying.
	\end{enumerate}
	If (one of) the three statements hold, we say that $h$ is \textbf{regularly varying with index} $\mathbf{\alpha}$.
\end{lemma}
The following proposition says that the condition (\ref{rv}) in Theorem \ref{isoasymptotics} is more general than the classical  assumption "$K(z)=k(|z|)$ with $k$ regularly varying with index $\beta\in(-d,0)$", discussed in Section \ref{secmotex}.

\begin{proposition}{\cite[Proposition 1.5.11]{Bingham}}\label{nonviceversa} If $k:\R_+\rightarrow\R$ is regularly varying with index $-\beta\in(-d,0)$ and locally bounded on $\R_+$, then 
	\[
	\frac{t^dk(t)}{\int_0^tk(r)r^{d-1}dr}\rightarrow(d-\beta)\quad \quad \text{as } t\rightarrow\infty
	\]
	and $w_t:=\omega_{d-1}\int_0^tk(r)r^{d-1}dr$ is regularly varying with index $\alpha=d-\beta$.
\end{proposition}

A fundamental tool for the proof of Theorem \ref{isoasymptotics} will be Potter's theorem (see e.g. \cite[Theorem 1.5.6]{Bingham}).
\begin{theorem}(Potter's Theorem)\label{potter} If $h:\R_+\rightarrow\R$ is regularly varying with index $\alpha\in\R$, then for every $A>1$, $\delta>0$, $\exists$ $X=X(h,A,\delta)>0$ such that
	$$|h(lt)/h(t)|\le A\max\{l^{\alpha+\delta},l^{\alpha-\delta}\} \quad\quad\forall l>\frac{X}{t}, \quad \forall t>X.$$
\end{theorem}
\begin{remark}{\textbf{(Justifying equation \eqref{K rrv ascov})}}\label{just}
	Note that if $k:\R_+\rightarrow \R$ is measurable, bounded and regularly varying with index $-\beta\in(-d,0)$, by Theorem \ref{potter} we have for $D,L\subseteq \R^d$ compact sets 
	\[
	\mathbf{1}_{\{|x-y|\ge X/t\}}\mathbf{1}_{D}(x)\mathbf{1}_{L}(y)\frac{k(t|x-y|)}{k(t)}\le c \,\mathbf{1}_{D}(x)\mathbf{1}_{L}(y)\max\left\{|x-y|^{-\beta-\delta}, |x-y|^{-\beta+\delta} \right\}
	\]
	where $X,c>0$ are suitable constants. Moreover, since $\beta\in(0,d)$, choosing $\delta>0$ small enough, the function on the RHS is integrable in $\R^{2d}$. Therefore, by dominated convergence theorem, we have as $t\rightarrow \infty$
	\[
	\int_D \int_L \mathbf{1}_{\{|x-y|\ge X/t\}}\frac{k(t|x-y|)}{k(t)}dxdy\rightarrow \int_D \int_L |x-y|^{-\beta}dxdy  .
	\]
	This last fact, together with
	\[
	\int_D \int_L \mathbf{1}_{\{|x-y|\le X/t\}}\frac{k(t|x-y|)}{k(t)}dxdy \le c \, \frac{1}{k(t)t^d}=\frac{1}{\ell(t)t^{d-\beta}} \rightarrow 0, \quad \quad \text{as $t\rightarrow\infty$,}
	\]
	concludes the proof of \eqref{K rrv ascov}.
\end{remark}
\begin{corollary}
	\label{coregvar}
	Consider $K:\R^d\rightarrow\R$ measurable, bounded function and $w:\R_+\rightarrow\R$ as defined in (\ref{w}). Assume that $w$ is regularly varying with index $\alpha>-1$. Then for every $U\in(0,\infty)$ 
	\[
	\int_0^U |w_{lt}/w_t|dl\rightarrow \int_0^U l^{\alpha} dl \quad \quad \text{as }t\rightarrow\infty.
	\]
\end{corollary}
\begin{proof}
	First of all, observe that by Potter's Theorem \ref{potter}, choosing $A=2$ and $\delta=\frac{\alpha+1}{2}>0$, $\exists$ $X=X(w,A,\delta)=X(K,\alpha)>0$ such that 
	\begin{equation}
		\label{BOO}
		|w_{lt}/w_t|\le 2\max\left\{l^{\frac{3\alpha+1}{2}},l^{\frac{\alpha-1}{2}}\right\} \quad\quad\forall l>\frac{X}{t}, \quad \forall t>X,
	\end{equation}
	with $l\mapsto\max\{l^{\frac{3\alpha+1}{2}},l^{\frac{\alpha-1}{2}}\}$ integrable on $[0,U]$ because $\alpha>-1$.
	Moreover, since $w$ is of the form (\ref{w}), for $0<l<X/t$ we have
	\[
	|w_{lt}/w_t|\le \frac{\max _{x\in\R^d}|K(x)|}{|w_t|}\omega_{d-1}X^{d} \quad \forall\, 0<l<X/t.
	\]
	Putting all together, we obtain 
	\[
	\int_0^U |w_{lt}/w_t|dl=\int_0^{X/t} |w_{lt}/w_t|dl+\int_{X/t}^U |w_{lt}/w_t|dl.
	\]
	Note that $w_t=\ell(t)t^{\alpha}$ with $\alpha>-1$, so $\ell(t)t^{1+\alpha}\rightarrow \infty$  and
	\[
	\int_0^{X/t} |w_{lt}/w_t|dl\le\frac{\max _{x\in\R^d}|K(x)|}{t|w_t|}\omega_{d-1}X^{d+1}=\frac{\max _{x\in\R^d}|K(x)|}{|\ell(t)t^{1+\alpha}|}\omega_{d-1}X^{d+1}\rightarrow 0 .
	\]
	Moreover, by (\ref{BOO}) and dominated convergence theorem we have
	\[
	\int_{X/t}^U |w_{lt}/w_t|dl\rightarrow\int_0^U l^{\alpha} dl,
	\]
	which concludes the proof.
\end{proof}

\section{Proof of Theorem \ref{isoasymptotics}}\label{exact asymptotics}
\begin{proof}(of Theorem \ref{isoasymptotics})
	By Proposition \ref{wellposedness} and the change of variable $x=x$, $y=x-z$, we have
	\begin{eqnarray}
		&&{\rm Cov}\left(\frac{\int_{tD}A_x dx}{t^{d/2}w_t^{1/2}}, \frac{\int_{tL}A_y dy}{t^{d/2}w_t^{1/2}}\right)\label{covcov}\\ &=&\int_{tD}\int_{tL}K(x-y)\frac{dxdy}{t^dw_t} \nonumber\\
		&=&\int_{\R^d}K(z)g_{tD,tL}(z)\frac{dz}{t^dw_t}\nonumber\\ &=&\int_{\{|z|\le {\rm diam}(D\cup L)t\}}K(z)g_{D,L}\left(\frac z t\right)\frac{dz}{w_t} \nonumber
	\end{eqnarray}
	where $g_{D,L}$ is defined in (\ref{g}), ${\rm diam}(D\cup L)$ is defined in (\ref{h}) and the last equality follows from (iii) of Proposition \ref{prop crosscov} and 
	\[
	g_{tD,tL}(z)={\rm Vol}(tD\cap (tL+z))=t^d{\rm Vol}\left(D\cap \left(L+\frac z t\right)\right)=t^dg_{D,L}\left(\frac z t\right).
	\]
	Passing to polar coordinates (if $d=1$, recall the notation (\ref{not1})), we have
	\begin{eqnarray*}
		&&\int_{\{|z|\le {\rm diam}(D\cup L)t\}}K(z)g_{D,L}\left(\frac z t\right)dz\\
		&&=\int_{S^{d-1}}d\theta \int_0^{{\rm diam}(D\cup L)t} dr\,  r^{d-1}K(r\theta)g_{D,L}\left(\frac{r}{t}\theta\right),
	\end{eqnarray*}
	and since $D,L\in\mathcal{D}$ have finite perimeter in both {\it \textbf{Case 1.}} (\ref{case1}) and {\it \textbf{Case 2.}} (\ref{case2}), by (v) in Proposition \ref{prop crosscov} the covariance (\ref{covcov}) becomes
	\begin{eqnarray*}
		&&{\rm Cov}\left(\frac{\int_{tD}A_x dx}{t^{d/2}w_t^{1/2}}, \frac{\int_{tL}A_y dy}{t^{d/2}w_t^{1/2}}\right)\\
		&=&w_t^{-1}\int_{S^{d-1}}d\theta\int_0^{{\rm diam}(D\cup L)} dl \left(-\frac{d}{dl}(g_{D,L}(l\theta))\right)\left(\int_0^{tl}  r^{d-1}K(r\theta)dr\right).
	\end{eqnarray*}
	\vspace{0.5mm}
	
	Now we distinguish the two cases in the statement of Theorem \ref{isoasymptotics}.\\ \\
	{\it \textbf{Case 1.}}
	If $K(z)=k(|z|)$ for some $k:\R_+\rightarrow\R$, then $K(r\theta)=k(r)$ and we have (recall the definition (\ref{w}) of $w$)
	\begin{eqnarray*}
		&&{\rm Cov}\left(\frac{\int_{tD}A_x dx}{t^{d/2}w_t^{1/2}}, \frac{\int_{tL}A_y dy}{t^{d/2}w_t^{1/2}}\right)\\
		&=&\frac{1}{\omega_{d-1}}\int_{S^{d-1}}d\theta\int_0^{{\rm diam}(D\cup L)} dl \left(-\frac{d}{dl}(g_{D,L}(l\theta))\right)\frac{w_{lt}}{w_t}
	\end{eqnarray*}\\ 
	{\it \textbf{Case 2.}}
	If $D,L$ are balls centered in the same point $x_0\in\R^d$, then $g_{D,L}(r\theta)=g_{D,L}(r\theta')$ for every $\theta,\theta'\in S^{d-1}$ and by a Fubini-Tonelli argument (and (v) in Proposition \ref{prop crosscov}) we have, for all $\theta'\in S^{d-1}$ 
	\begin{eqnarray}\label{recurrent}
		&&{\rm Cov}\left(\frac{\int_{tD}A_x dx}{t^{d/2}w_t^{1/2}}, \frac{\int_{tL}A_y dy}{t^{d/2}w_t^{1/2}}\right)\\&=&\int_0^{{\rm diam}(D\cup L)} \left(-\frac{d}{dl}(g_{D,L}(l\theta'))\right)\frac{w_{lt}}{w_t} dl\nonumber\\
		&=&\frac{1}{\omega_{d-1}}\int_{S^{d-1}}d\theta\int_0^{{\rm diam}(D\cup L)}dl\left(-\frac{d}{dl}(g_{D,L}(l\theta))\right)\frac{w_{lt}}{w_t} \nonumber
	\end{eqnarray}which is equal to what we obtained in Case 1.\\ 
	
	Since the expression for the covariance is the same in both cases, to prove (\ref{covgen}) (and conclude the proof of Theorem \ref{isoasymptotics}) we only need to show 
	\begin{eqnarray}
		&&\int_{S^{d-1}}d\theta\int_0^{{\rm diam}(D\cup L)}dl\left(-\frac{d}{dl}(g_{D,L}(l\theta))\right)\frac{w_{lt}}{w_t}  \nonumber\\ &&\rightarrow\int_{S^{d-1}}d\theta\int_0^{{\rm diam}(D\cup L)}dl\left(-\frac{d}{dl}(g_{D,L}(l\theta))\right)l^{\alpha}  \quad \text{ as }t\rightarrow\infty \label{finalissima},
	\end{eqnarray}
	where $\alpha$ is the index of regular variation of $w_t$. First of all, note that $w_{lt}/w_t\rightarrow l^{\alpha}$ for every $l>0$, because $w$ is regularly varying with index $\alpha$, implying the point-wise convergences (for almost every $l>0,\theta \in S^{d-1}$)
	\[
	F_t(l,\theta):=\left(-\frac{d}{dl}(g_{D,L}(l\theta))\right)\frac{w_{lt}}{w_t}\rightarrow F(l,\theta):=\left(-\frac{d}{dl}(g_{D,L}(l\theta))\right)l^{\alpha},\quad \quad \text{as $t\rightarrow\infty$}
	\]
	and
	\[
	M_t(l):={\rm Per}(D)\wedge {\rm Per}(L)\,|w_{lt}/w_t|\rightarrow M(l):={\rm Per}(D)\wedge {\rm Per}(L)\,l^{\alpha},\quad \quad \text{as $t\rightarrow\infty$}.
	\]
	Moreover, by Proposition \ref{prop crosscov} the inequality $$|F_t(l,\theta)|=\left|\left(-\frac{d}{dl}(g_{D,L}(l\theta))\right)\frac{w_{tl}}{w_t}\right|\le M_t(l):={\rm Per}(D)\wedge {\rm Per}(L)\,|w_{lt}/w_t|$$ holds for almost every $l>0,\theta \in S^{d-1}$, and by Corollary \ref{coregvar} we have 
	\[
	\int_{S^{d-1}}d\theta\int_{0}^{{\rm diam}(D\cup L)}dl\,M_t(l)\rightarrow \int_{S^{d-1}}d\theta\int_{0}^{{\rm diam}(D\cup L)}dlM(l).
	\]
	Putting all together, by the generalized dominated convergence theorem we obtain
	\[
	\int_{S^{d-1}}d\theta\int_{0}^{{\rm diam}(D\cup L)}dl\,F_t(l, \theta)dl\rightarrow \int_{S^{d-1}}d\theta\int_{0}^{{\rm diam}(D\cup L)}dl\, F(l,\theta),\quad \quad \text{as $t\rightarrow\infty$},
	\]
	which is exactly \eqref{finalissima}.
\end{proof}

\section{Non-linear functionals of stationary Gaussian fields}\label{nnfGF}
In this section, we show how Theorem \ref{isoasymptotics} can be applied in the setting of non-linear functionals of stationary Gaussian fields to obtain multi-dimensional limit theorems. Through several examples, we will also compare the results obtained to the ones found in the existing literature.\\

Let $B=(B_x)_{x\in\R^d}$ be a continuous, centered, stationary Gaussian field with $B_x\sim N(0,1)$ $\forall$ $x\in\R^d$ and (continuous) covariance function 
\begin{equation}
	\label{Cdef}
	C:\R^d\rightarrow\R, \quad \quad{\rm Cov}(B_x,B_y)=C(x-y),\quad|C(z)|\le C(0)=1.
\end{equation}
Denote again by $\mathcal{D}$ a collection of compact sets in $\R^d$ and consider the non-linear functional\footnote{(\ref{set indexed phi}) is always well-posed in the $L^2$-sense, see e.g.  \cite[Proposition 3]{MN}.} of $B$ 
\begin{equation}
	\label{set indexed phi}
	\int_{tD}\varphi(B_x)dx,\quad \quad D\in\mathcal{D},\quad t>0,
\end{equation}
where $tD:=\{tx: x\in D\}$, $\varphi:\R\rightarrow\R$ is not constant (to avoid trivialities), $\varphi\in L^2(\R,\gamma(dx))$, $\int_{\R}\varphi(x)\gamma(dx)=0$ and $\gamma$ is the standard Gaussian measure on $\R$. Then, we consider the $L^2$-decomposition of $\varphi$  (see e.g. \cite[Section 1.4]{bluebook})
\[
\varphi=\sum_{q=1}^\infty a_qH_q,
\]
where $H_q$ is the $q$-th Hermite polynomial and
\begin{equation}
	\label{normphi}
	\|{\varphi}\|^2_{L^2(\R,\gamma(dx))}=\sum_{q=1}^\infty q!a_q^2<\infty.
\end{equation}
Therefore, $A_x:=\varphi(B_x)$ can be expressed as 
\[
A_x:=\varphi(B_x)=\sum_{q=R}^\infty a_qH_q(B_x),\quad \quad a_R\neq0,
\]
where $R=\inf\{q\ge1 :a_q\neq0\}<\infty$ denotes the \textbf{Hermite rank} of $\varphi$ and the equality holds in $L^2(\Omega)$ sense. Moreover, we have the isometry property (see e.g. \cite[Section 1.4]{bluebook})
\[
\E[H_q(B_x)H_r(B_y)]=q!C^q(x-y) \delta_{qr}
\]
where $\delta_{qr}$ is the Kronecker delta.
As a consequence, the covariance function $K(x-y)={\rm Cov}(A_x,A_y)$ of  $A=(A_x)_{x\in\R^d}$ is
\begin{equation}
	\label{exp K}
	K(z)=\sum_{q=R}^\infty q!a_q^2C^q(z).
\end{equation}
Note that $A$ is obviously weakly stationary and $K$ is continuous, because uniform limit of continuous functions (thanks to (\ref{normphi}) and (\ref{Cdef})). Moreover, $w$ in (\ref{w}) in this case is
\begin{equation}
	\label{expansion w}
	w_t=\int_{\{|z|\le t\}}K(z)dz=\sum_{q=R}^\infty a_q^2w_{q,t}, \quad \quad t>0,
\end{equation}
where we additionally introduced the notation 
\begin{equation*}
	\label{Wq}
	w_{q,t}=q!\int_{\{|z|\le t\}}C^q(z)dz.
\end{equation*}
Now we can finally understand what new results we obtain in this setting applying Theorem \ref{isoasymptotics}, dividing the study in different cases.
\subsection{The Breuer-Major case.}\label{sectionBM}
If $C\in L^R(\R^d)$, we are in the Breuer-Major case. Since $C\in L^R(\R^d)$ implies $K\in L^1(\mathbb{R}^d)$ (see (\ref{Cdef}) and (\ref{exp K})), reasoning as in (\ref{Kintegrable ascov}) we get
\begin{equation}
	\label{bmcov}
	\lim_{t\rightarrow\infty}{\rm Cov}\left(\frac{\int_{tD}\varphi(B_x) dx}{t^{d/2}}, \frac{\int_{tL}\varphi(B_y) dy}{t^{d/2}}\right)= {\rm Vol}(D\cap L)\,\int_{\R^d}K(z)dz.
\end{equation}
Note that if $D$ is fixed and $u,v\in[0,1]$, we have
\[
\lim_{t\rightarrow\infty}{\rm Cov}\left(\frac{\int_{tD u^{1/d}}\varphi(B_x) dx}{t^{d/2}}, \frac{\int_{tD v^{1/d}}\varphi(B_y) dy}{t^{d/2}}\right)= {\rm Vol}(D)\,\min\{u,v\}\int_{\R^d}K(z)dz,
\]
which is, up to a scaling factor, the covariance function of a Brownian motion (see also \eqref{bmmultispecialcase}).
Moreover, we have the following fundamental result, proved for the first time in the discrete setting by Breuer and Major in their seminal paper \cite{BM}, and extended to different settings by several authors, see e.g. \cite{Ben Hariz}, \cite{CNN}, \cite{IL}, \cite{LO}, \cite{NP}, \cite{NPP}, \cite{NZ}.

\begin{theorem}(Breuer-Major)\label{thm:BM}
	If $C\in L^R(\R^d)$, then 
	\begin{equation}
		\label{bmeq}
		t^{-d/2}\int_{tD}\varphi(B_x)dx \overset{\rm law}{\to} N(0,\sigma^2{\rm Vol}(D))
	\end{equation}
	where 
	\begin{equation*}\label{sigma}
		\sigma^2=\int_{\R^d}K(z)dz=\sum_{q=R}^\infty q!a_q^2\int_{\R^d}C(z)^qdz\geq 0.
	\end{equation*}
\end{theorem}
Note that also multi-dimensional (see e.g. \cite[Theorem 1.2]{NZ}) and stronger type (see e.g. \cite[Theorem 1.1]{CNN}) central limit theorems of the form
\begin{equation}\label{bmmulti}
	\left(t^{-d/2}\int_{tD}\varphi(B_x)dx\right)_{D \in\mathcal{D}'} \overset{\rm f.d.d.}{\to} \sigma G=(\sigma G(D))_{D\in\mathcal{D}'}
\end{equation}
have been proved in the literature, where $G$ is a Gaussian noise, with $${\rm Cov}(G(D),G(L))= {\rm Vol}(D\cap L),$$ and $\mathcal{D}'$ is a suitable class of compact sets. In particular, considering $\mathcal{D}'=\{D u^{1/d}, u\in[0,1]\}$ with fixed $D=\{x\in\R^d : |x|\le 1\}$, we have (see \cite[Theorem 2.3.1]{IL})
\begin{equation}\label{bmmultispecialcase}
	\left(t^{-d/2}\int_{tD{u^{1/d}}}\varphi(B_x)dx\right)_{u\in[0,1]} \overset{\rm f.d.d.}{\to} \sigma {\rm Vol}(D) W=\left(\sigma {\rm Vol}(D) W_u\right)_{u\in[0,1]},
\end{equation}
where $W$ is a standard Brownian motion. 

If $\sigma^2>0$, then (\ref{rv}) holds with $\alpha=0$ and Theorem \ref{isoasymptotics} implies
\begin{equation}\label{nonbmcov}
	\lim_{t\rightarrow\infty}{\rm Cov}\left(\frac{\int_{tD}\varphi(B_x) dx}{t^{d/2}w_t^{1/2}}, \frac{\int_{tL}\varphi(B_y) dy}{t^{d/2}w_t^{1/2}}\right)= {\rm Vol}(D\cap L),
\end{equation}
which is (\ref{bmcov}). Moreover, Theorem \ref{thm:BM} (and its multi-dimensional generalizations) ensure Gaussian fluctuations, which are not guaranteed by Theorem \ref{isoasymptotics}. Nevertheless, our Theorem \ref{isoasymptotics} allows to obtain the asymptotic covariances (\ref{nonbmcov}) in many other situations (see e.g. Example \ref{Berry4} and \ref{Berry3}), whenever (\ref{rv}) holds with $\alpha=0$, i.e. $w_t$ is slowly varying.

Conversely, if $\sigma^2=0$, (\ref{bmcov}) and (\ref{bmeq}) only imply that  $t^{d/2}$ is not the correct normalization if we hope in a non-degenerate limit in distribution (see also the discussion in Section \ref{secmotex}). However, $w_t$ could be regularly varying, allowing us to apply Theorem \ref{isoasymptotics}. This is exactly what happens in the following example.

\begin{example}{\textbf{(Breuer-Major case, $\mathbf{\sigma^2=0}$.)}} \label{ex:bm0}
	Fix  $d=1$, consider a fractional Brownian motion $(W_x^H)_{x\in\R}$ with Hurst index $H\in(0,1/2)$ and recall its covariance structure
	\[
	{\rm Cov}(W_x^H,W_y^H)=\frac{1}{2}\left(|x|^{2H}+|y|^{2H}-|x-y|^{2H}\right).
	\]
	Furthermore, consider the associated continuous, stationary Gaussian process $B=(B_x)_{x\in\R^d}$, defined by
	\[
	B_x:=W^H_{x+1}-W^H_x .
	\]
	Note that we consider the $1$-increment process $B$ instead of $W^H$ because $B$ is stationary (hence weakly stationary, and we can apply Theorem \ref{isoasymptotics}), while $W^H$ is not even weakly stationary (it only has stationary increments).
	A standard computation yields that the covariance function $C:\R\rightarrow\R$ of $B$ is given by
	\[
	C(z)=\frac 12\left(|1+z|^{2H}+|1-z|^{2H}-2|z|^{2H}\right)
	\]
	and $C\in L^1(\R^d)$, because as $|z|\rightarrow\infty$ we have (by Taylor expansion)
	\[
	C(z)\sim 2H(2H-1)|z|^{2H-2}, \quad \quad 2H-2\in(-2,-1).
	\]
	If we choose $\varphi(x)=x$, then $\int_{tD}\varphi(B_x)dx=\int_{tD}B_x dx$ is Gaussian for every $t>0$. The important fact to note in this example is that 
	\[
	\sigma^2=\int_{\R^d}K(z)dz=\int_{\R^d}C(z)dz=0,
	\]
	implying that (\ref{bmcov}) and the Breuer-Major theorem only say that $t^{d/2}=t^{1/2}$ is not the correct normalization. Indeed, as $t\rightarrow\infty$ we have
	\begin{eqnarray*}
		w_t&=&w_{1,t}=2\int_0^t C(z)dz=\int_0^t |1+z|^{2H}+|1-z|^{2H}-2|z|^{2H}dz\\
		&=&\frac{1}{2H+1}\left((t+1)^{2H+1}+(t-1)^{2H+1}-2t^{2H+1}\right)\sim 2H t^{2H-1},
	\end{eqnarray*}
	which means that $w_t$ is regularly varying with index $\alpha=2H-1\in(-1,0)$. Therefore, applying Theorem \ref{isoasymptotics} we obtain that the correct normalization is $t^{1/2}w_t^{1/2}\sim \sqrt{2H} t^{H}$ and $\forall D,L\in\mathcal{D}$
	\[
	{\rm Cov}\left(\frac{\int_{tD}B_x dx}{t^{1/2}w_t^{1/2}}, \frac{\int_{tL}B_y dy}{t^{1/2}w_t^{1/2}}\right)\rightarrow \frac{1}{2}\int_{0}^\infty -\frac{d}{dl}(g_{D,L}(l)+g_{D,L}(-l)) l^{2H-1} dl
	\]
	which, if $D=[0,r]$ and $L=[0,s]$, is the covariance function of a fractional Brownian motion with Hurst index $H\in(0,1/2)$ (see Remark \ref{comparison}).
\end{example}

\subsection{The long memory case.}\label{sectionlongshort}
Another frequent assumption is $C(z)=\rho(|z|)$ radial with $\rho$ regularly varying with index $-\beta\in(-d/R,0)$ (see e.g. \cite{L} for some examples), that is:
\[
\rho(r)=\frac{\ell(r)}{r^{\beta}}, \quad \quad r>0, 
\]
with $\ell:\R_+\rightarrow\R$ slowly varying (see Section \ref{regularly varying section}).
In this case, we say that we are in the long-memory case. As $|z|\rightarrow\infty$, we have
\[
K(z)=k(|z|)=\sum_{q=R}^\infty q!a_q^2\rho^q(|z|)\sim R! a_R^2\rho^R(|z|)=R! a_R^2 \frac{\ell^R(|z|)}{|z|^{R\beta}},
\]
implying that $k:\R_+\rightarrow\R$ is regularly varying with index $-R\beta\in(-d,0)$. Therefore, reasoning as in (\ref{K rrv ascov}) we obtain
\begin{eqnarray}
	\lim_{t\rightarrow\infty}{\rm Cov}\left(\frac{\int_{tD}\varphi(B_x) dx}{t^{d}k(t)^{1/2}}, \frac{\int_{tL}\varphi(B_y) dy}{t^{d}k(t)^{1/2}}\right)&=& \int_{\R^d} g_{D,L}(z)|z|^{-R\beta}dz\nonumber \\
	&=&\int_{D}\int_L|x-y|^{-R\beta}dxdy,\label{dmcov}
\end{eqnarray}
\noindent
and a correct normalization turns out to be $$t^{d}k(t)^{1/2}\sim\sqrt{R!\,a_R^2}\,t^{d-\frac{R\beta}{2}}\ell^{R/2}(t), \quad \quad \text{as }t\rightarrow\infty,$$i.e. the one usually observed in the long memory context (see e.g. \cite{LO}).
Note that $C\notin L^R(\R^d)$, so we are not in the Breuer-Major case, and 
\begin{equation}
	\label{uffa2}
	w_t= \sum_{q=R}^\infty a_q^2w_{q,t}\sim a_R^2w_{R,t}\quad \quad \text{as $t\rightarrow\infty$},
\end{equation}
which by Proposition \ref{nonviceversa} is regularly varying with index $\alpha=d-R\beta \in(0,d)$. For this reason, we can apply Theorem \ref{isoasymptotics}, obtaining
\begin{eqnarray*}
	\lim_{t\rightarrow\infty}{\rm Cov}\left(\frac{\int_{tD}\varphi(B_x) dx}{t^{d/2}w_t^{1/2}}, \frac{\int_{tL}\varphi(B_y) dy}{t^{d/2}w_t^{1/2}}\right)=\frac{\alpha}{\omega_{d-1}}\int_{D}\int_L|x-y|^{\alpha-d}dxdy,
\end{eqnarray*}
which is exactly (\ref{dmcov}) (indeed, $w_{t}\sim \frac{\omega_{d-1}}{\alpha}t^{d}k(t)$ by  Proposition \ref{nonviceversa}).

The first limit theorems in the long memory context were proved in the seminal works \cite{DM}, \cite{Rosenblatt}, \cite{Taqqu} and then extended in many directions. What happens in this case is that (\ref{set indexed phi}) (suitably normalized by $t^d k(t)$) has asymptotically an Hermite distribution of order $R$, which is not Gaussian if $R\ge2$. Moreover, the result can be extended to multi-dimensional (and stronger type) central limit theorems, see e.g. \cite{LO}. 

Summarizing, if $B$ has a radial covariance function $C(z)=\rho(|z|)$ with $\rho$ regularly varying with index $-\beta\in(-d/R,0)$, then we can compute the asymptotic covariances of $B$ using (\ref{dmcov}) or Theorem \ref{isoasymptotics} (since $w_t$ is regularly varying with index $\alpha=d-R\beta$). Moreover, the forementioned limit theorems apply, yielding an Hermite limiting distribution. However, Theorem \ref{isoasymptotics} allows to obtain the asymptotic covariances in many other situations in which (\ref{dmcov}) does not hold, whenever $w_t$ is regularly varying with index $\alpha>0$ and $k$ is not regularly varying (see Example \ref{Berry2}).

\subsection{The Berry's case.}\label{sectionberry}
The Berry's random wave model is a Gaussian field which arises as the local scaling limit of a variety of random fields on two-dimensional manifolds, see \cite{xxx} for the details. It was first introduced by Berry in \cite{Berry}, and is used in quantum mechanics to model the Laplace eigenfunctions of classically chaotic billiards with large eigenvalue (see also \cite{NPR} and the references therein). It is defined as a smooth, stationary, centered Gaussian field $B$ on the plane ($d=2$) with radial covariance function $$C(z)=\rho(|z|)=J_0(|z|),$$ where $J_0$ is the Bessel function of the first kind of order $0$  (see e.g. \cite{krasikov}), with series expansion
\[
J_0(r)=\sum_{j=0}^\infty\frac{(-1)^{j}}{(j!)^2}\left(\frac{r}{2}\right)^{2j}
\]
and asymptotic behavior
\begin{equation}\label{asympBerry}
	J_0(r)=\sqrt{\frac{2}{\pi}}\,r^{-\frac12}\,\cos\left(r-\frac{\pi}4\right) + O\left({r^{-\frac32}}\right)\quad \mbox{as $r\to\infty$}.
\end{equation}
Note that $\rho^q$ is not regularly varying. This means that we can not use Proposition \ref{nonviceversa} to prove that $w_{q,t}$ is regularly varying, as we did in Section \ref{sectionlongshort}. Nevertheless, this latter fact is  given by the following lemma (for $q\neq 1$) and is a consequence of the results in \cite{GMT}. 
\begin{remark}\label{rem:w1}
	For $q=1$, $w_{1,t}$ is not regularly varying. To prove this fact, consider $J_1$, the Bessel function of the first kind of order $1$ (see e.g. \cite{krasikov}), with series expansion
	\[
	J_1(r)=\left(\frac{r}{2}\right)\sum_{j=0}^\infty\frac{(-1)^{j}}{j!(j+1)!}\left(\frac{r}{2}\right)^{2j}
	\]
	and asymptotic behavior
	\[
	J_1(r)=\sqrt{\frac{2}{\pi}}\,r^{-\frac12}\,\cos\left(r-\frac{3}4 \pi\right) + O\left({r^{-\frac32}}\right)\quad \mbox{as $r\to\infty$}.
	\]
	Note that using the series expansions of $J_0$ and $J_1$ we have
	\[
	\frac{d}{dr}\left( J_1(r)r \right)=J_0(r)r. 
	\]
	Therefore, since
	\[
	w_{1,t}=2\pi\int_0^t J_0(r)r \,dr= J_1(t)t=2\pi\sqrt{\frac{2}{\pi}}\,t^{\frac12}\,\cos\left(t-\frac{3}4 \pi\right) + O\left({t^{-\frac12}}\right)\quad \mbox{as $t\to\infty$},
	\]
	$w_{1,t}$ is not a regularly varying function, and Theorem \ref{isoasymptotics} can not be applied to compute the asymptotic covariances of $\int_{tD}H_1(B_x)dx=\int_{tD}B_x dx$ and $\int_{tL}H_1(B_x)dx=\int_{tL}B_x dx$ (recall that $H_1(x)=x$ is the first Hermite polynomial). 
	
	Nevertheless, with a different argument, something can be said about the variance of $\int_{tD}B_x dx=\int_{tD}H_1(B_x)dx$. Indeed, we have the formula (see e.g. \cite{Schoenberg})
	\[
	C(z)=J_0(|z|)=\frac{1}{\omega_{d-1}}\int_{S^{d-1}} e^{i\langle x,\theta \rangle}d\theta.
	\]
	and after applying Fubini-Tonelli theorem we obtain
	\[
	{\rm Var}\left(\int_{tD} B_x dx \right)=\int_{tD} \int_{tD} J_0(|x-y|)dxdy=\frac{1}{\omega_{d-1}}\int_{S^{d-1}} |\mathcal{F}[\mathbf{1}_{tD}]|^2 (\theta) d\theta=\frac{t^{2d}}{\omega_{d-1}}\int_{S^{d-1}} |\mathcal{F}[\mathbf{1}_D]|^2 (t\theta) d\theta,
	\]
	where $\mathcal{F}$ denotes the Fourier transform.
	Therefore, if $|\mathcal{F}[\mathbf{1}_D](x)|=o(|x|^{-d/2})$, we have
	\[
	{\rm Var}\left(\int_{tD} B_x dx \right)=o(t^{d}), \quad \text{as }t\rightarrow\infty.
	\]
	This fact will be needed in Example \ref{Berry1}.
\end{remark}
\begin{lemma}\label{ratesberry}
	Let $(B_x)_{x\in\R^2}$ be the Berry's random wave model, and recall the notation introduced in the first part of Section \ref{nnfGF}. Then for $q\ge2$, as $t\rightarrow\infty$, we have
	\begin{equation*}\label{rate}
		w_{q,t} \sim \left\{\begin{array}{lll}
			c\,t&&\mbox{if $q=2$}\\
			c \log(t)&&\mbox{if $q=4$}\\
			c&&\mbox{if $q=3$ or  $q\ge5$}
		\end{array}
		\right. ,
	\end{equation*} 
	where $c\in(0,\infty)$ depends on $q$, but does not depend on $t$.
\end{lemma}
\begin{proof}
	In the sequel, let $c$ be a positive constant which may change value.
	For $q=2$ even, by (\ref{asympBerry}) and the trigonometric identity $2 \cos^2(x)=1+\cos(2x)$ we have
	\begin{align*}
		w_{2,t}&=2\pi\int_0^t J_0^2(r)r\,dr=c\,\left(\int_1^t {\cos^{2}(r-\pi/4)}dr+O\left(\int_0^t \frac{dr}{r} \right) \right)\\
		&= c\,\left(t-1+\int_1^t {\cos(2r-\pi/2)}dr+O\left(\int_0^t \frac{dr}{r} \right) \right)\\
		&= c\,\left( t+O(\log(t))\right)\sim c t.
	\end{align*}
	For $q=4$, by (\ref{asympBerry}) and the trigonometric identity $8 \cos^4(x)=2(1+\cos(2x))^2=2+4\cos(2x)+2\cos^2(2x)=2+4\cos(2x)+1+\cos(4x)$, we have
	\begin{align*}
		w_{4,t}&=2\pi\int_0^t J_0^4(r)r\,dr=c\,\left(\int_1^t \frac{\cos^{4}(r-\pi/4)}{r}dr+O\left(\int_0^t \frac{dr}{r^2} \right) \right)\\
		&= c\,\left(3\log(t)+4\int_1^t {\cos(2r-\pi/2)}dr+\int_1^t {\cos(4r-\pi)}dr+O\left(\int_0^t \frac{dr}{r^2} \right) \right)\\
		&= c\,\left( \log(t)+O(1)\right)\sim c \log(t).
	\end{align*}
	If $q\ge 6$ is even, $w_{q,t}>0$, $w_{q,t}$ increasing in $t$, and $w_{q,t}=O(1)$ (see Remark \ref{remainderrem}).
	Therefore, we have $w_{q,t}\sim c>0$ for some positive constant $c>0$.
	
	If $q\ge 3$ is odd, the fact that $w_{q,t}$ converges to a positive constant is more difficult to see, and is proved in \cite{GMT} by means of Pearson's random walks. 
\end{proof}

\begin{remark}\label{remainderrem}
	Note that since $|C(z)|=|J_0(|z|)|\le 1$ and (\ref{asympBerry}) holds, we have 
	\[
	\left|J_0^5(r)\right|\le c\,\left( 1\,\wedge\frac{1}{r^{5/2}}\right), \quad r>0,
	\]
	where $c$ is a positive constant. Moreover, since $|C(z)|=|J_0(|z|)|\le 1$ we have
	\[
	\int_{\R^d}\left|C^q(z)\right| dz\le \int_{\R^d}\left|C^5(z)\right|dz=2\pi\int_0^\infty \left|J_0^5(r)\right| r \,dr\le c\,\int_1^\infty \frac{dr}{r^{3/2}}<\infty,\quad \quad q\ge5.
	\]
	Therefore, we obtain the uniform bound
	\begin{equation}
		\label{boundreminder}
		\sum_{q=5}^\infty a_q^2 w_{q,t}\le  \|{\varphi}\|^2_{L^2(\R_+,\gamma(dx))}\int_{\R^d}\left|C^5(x)\right|dx<\infty.
	\end{equation} 
\end{remark}

Combining Lemma \ref{ratesberry} with Theorem \ref{isoasymptotics}, we are now able to compute the asymptotic covariances
\[
\lim_{t\rightarrow\infty}{\rm Cov}\left(\frac{\int_{tD}\varphi(B_x) dx}{t^{d/2}w_t^{1/2}}, \frac{\int_{tL}\varphi(B_y) dy}{t^{d/2}w_t^{1/2}}\right),
\]
where $B=(B_x)_{x\in\R^2}$ is the Berry's random wave model, 
for a large class of functions $\varphi$ and compact domains $D,L$. If we further apply \cite[Theorem 2]{MN} and the Peccati-Tudor multi-dimensional fourth moment theorem (see \cite{PT}), we are able to prove multi-dimensional central limit theorems for functionals of the Berry's random wave model. All these facts are explained  in the upcoming Examples \ref{Berry5}-\ref{Berry1}. Before we move on to the latter, to  ease the exposition, let us state a simplified version of the Peccati-Tudor theorem in our context.

\begin{theorem}\cite[Proposition 1]{PT} \label{PTtheo}
	Let $B=(B_x)_{x\in\R^d}$ be a continuous, centered, stationary Gaussian field with covariance function $C$. Let $\mathcal{D'}$ be a class of compact sets in $\R^d$.  Assume that the following three conditions hold:
	\begin{enumerate}
		\item ${\rm Var}\left(\int_{tD}H_q(B_x)dx\right)>0$, for every $D\in\mathcal{D'}$ and $t$ large enough.
		
		\item For every $D\in\mathcal{D'}$ 
		$$ \frac{\int_{tD}H_q(B_x)dx}{\sqrt{{\rm Var}\left(\int_{tD}H_q(B_x)dx\right)}}\overset{law}{\rightarrow}N,$$  where $N\sim N(0,1)$ has a standard Gaussian distribution.
		
		\item For every $D,L\in\mathcal{D'}$, we have
		\[
		{\rm Cov}\left(\frac{\int_{tD}H_q(B_x)dx}{\sqrt{{\rm Var}\left(\int_{tD}H_q(B_x)dx\right)}},\frac{\int_{tL}H_q(B_y)dy}{\sqrt{{\rm Var}\left(\int_{tL}H_q(B_y)dy\right)}}\right)\rightarrow b(D,L),
		\]
		where $b(D,L)$ is a constant depending only on $D$ and $L$.
	\end{enumerate}
	Then, as $t\rightarrow\infty$ we have the multi-dimensional central limit theorem
	\[
	\left(\frac{\int_{tD}H_q(B_x)dx}{  \sqrt{{\rm Var}\left(\int_{tD}H_q(B_x)dx\right)} }\right)_{D\in\mathcal{D'}}\overset{f.d.d.}{\rightarrow} G=\left( G(D)\right)_{D\in\mathcal{D'}}
	\]
	where $G$ is a centered, set-indexed Gaussian field with covariances
	\[
	{\rm Cov}(G(D),G(L))=b(D,L).
	\]
\end{theorem}

\begin{example}{\textbf{(Berry with $R\ge5$.)}}\label{Berry5} If $R\ge5$, then $C\in L^R(\R^2)$ and we are in the Breuer-Major case of Section \ref{sectionBM}. In particular, by Lemma \ref{ratesberry} we have that $\sigma^2>0$ and $w_t$ is slowly varying (see (\ref{boundreminder})). Therefore, the asymptotic covariances are given by (\ref{bmcov}) or (applying Theorem  \ref{isoasymptotics}) by (\ref{nonbmcov}), and Gaussian fluctuations follow by multi-dimensional Breuer-Major theorems of the form (\ref{bmmulti})-(\ref{bmmultispecialcase}).
\end{example}
\begin{example}{\textbf{(Berry with $R=4$.)}}\label{Berry4}
	When $R=4$, note that $K(z)=k(|z|)$ is radial, non-integrable and $k$ is not regularly varying, implying that we are not in the classical cases discussed in Sections \ref{sectionBM} and \ref{sectionlongshort}. Despite this, $w_t$ is slowly varying (see Lemma \ref{ratesberry}, (\ref{expansion w}) and (\ref{boundreminder})) and in particular
	\[
	w_t\sim a_4^2w_{4,t}\sim c \,\log(t) \quad \quad \text{as }t\rightarrow\infty,
	\] 
	where $c$ is a positive constant.
	Then, Theorem \ref{isoasymptotics} can be applied to obtain the asymptotic covariances
	\begin{equation}
		\label{asympcovBerry4}
		\lim_{t\rightarrow\infty}{\rm Cov}\left(\frac{\int_{tD}\varphi(B_x) dx}{t^{d/2}w_t^{1/2}}, \frac{\int_{tL}\varphi(B_y) dy}{t^{d/2}w_t^{1/2}}\right)={\rm Vol}(D\cap L),
	\end{equation}
	for every $D,L\in\mathcal{D}$, where $\mathcal{D}$ is the class of compact sets with finite perimeter introduced in \eqref{case1}. 
	Moreover, by using reduction techniques (see e.g. the proof of \cite[Proposition 4]{MN}) in (\ref{Berry4conv}) and the spectral CLT \cite[Theorem 2]{MN} in (\ref{Berryconv4bis}), we have for $D\subseteq \R^2$ compact with ${\rm Vol}(D)>0$:
	\begin{equation}
		\label{Berry4conv}
		\E\left[\left(\frac{\int_{tD}\varphi(B_x)dx}{\sqrt{{\rm Var}\left(\int_{tD}\varphi(B_x)dx\right)}}-sgn(a_4)\frac{\int_{tD}H_4(B_x)dx}{\sqrt{{\rm Var}\left(\int_{tD}H_4(B_x)dx\right)}}\right)^2\right]\rightarrow0
	\end{equation}
	and 
	\begin{equation}
		\label{Berryconv4bis}
		\frac{\int_{tD}H_4(B_x)dx}{\sqrt{{\rm Var}\left(\int_{tD}H_4(B_x)dx\right)}}\overset{law}{\rightarrow}N(0,1).
	\end{equation}
	Combining (\ref{asympcovBerry4}), (\ref{Berry4conv}), (\ref{Berryconv4bis}) and Theorem \ref{PTtheo} with $\mathcal{D}'=\mathcal{D}$, we obtain
	\begin{equation}
		\label{NICE}
		\left(\frac{\int_{tD}\varphi(B_x)dx}{t^{d/2}w_t^{1/2}}\right)_{D\in\mathcal{\mathcal{D}}}\overset{f.d.d.}{\to}(G(D))_{D\in\mathcal{\mathcal{D}}},
	\end{equation}
	where $G=(G(D))_{D\in\mathcal{D}}$ is a Gaussian noise with covariances given by (\ref{asympcovBerry4}). Note that we have proved \eqref{NICE} on $\mathcal{D}$, excluding the compact sets with infinite perimeter, since in the latter cases Theorem \ref{isoasymptotics} can not be applied as in \eqref{asympcovBerry4} to verify condition (iii) of Theorem \ref{PTtheo}.
	Note also that (\ref{NICE}) could have been obtained (for a less general class of domains $\mathcal{A}$ defined in \cite{PV}) combining the results proved in \cite{PV} for the nodal lengths of the Berry's random wave model and the reduction principle in \cite{Vidotto}. 
\end{example}
\begin{example}{\textbf{(Berry with $R=3$.)}} \label{Berry3} If $a_4\neq0$, then $w_t\sim a_4^2w_{4,t}$ and we can reason exactly as in Example \ref{Berry4}. So we can assume that $R=3$ and $a_4=0$. If this is the case, by Lemma \ref{ratesberry}, (\ref{expansion w}) and (\ref{boundreminder}) we have that $w_t$ converges to a positive constant, so that it is slowly varying.  Then, by Theorem \ref{isoasymptotics} we have for every $D,L$ in the class $\mathcal{D}$ of compact sets with finite perimeter  (see \eqref{case1})
	\[\lim_{t\rightarrow\infty}{\rm Cov}\left(\frac{\int_{tD}\varphi(B_x) dx}{t^{d/2}w_t^{1/2}}, \frac{\int_{tL}\varphi(B_y) dy}{t^{d/2}w_t^{1/2}}\right)={\rm Vol}(D\cap L).\]
	Note that in this case $C\notin L^3(\R^d)$, so we are not in the Breuer-Major case, where the covariances can be obtained for $D,L$ compact domains as in (\ref{bmcov}). In fact, in this case $\int_{\R^d}C^3(z)dz$ is only conditionally convergent, but not absolutely convergent. Nevertheless, excluding compact domains with infinite perimeter, we can compute the asymptotic covariances using Theorem \ref{isoasymptotics}.
	
	To the best of our knowledge, limit theorems in the case $R=3$ and $a_4=0$ have never been proved. 
\end{example}
\begin{example}{\textbf{(Berry with $R=2$.)}} \label{Berry2}When $R=2$, we have (by Lemma \ref{ratesberry}, (\ref{expansion w}) and (\ref{boundreminder})), $c$ being a positive constant,
	\[
	w_t\sim a_2^2w_{2,t}\sim c\, t \quad \quad \text{as }t\rightarrow\infty.
	\]
	Again, we are not in the classical cases discussed in Section \ref{sectionBM} or \ref{sectionlongshort}, but $w_t$ is regularly varying with index $\alpha=1$. Then, by Theorem \ref{isoasymptotics} we have for every $D,L$ in the class $\mathcal{D}$ of compact sets with finite perimeter  (see (\ref{case1}))
	\begin{equation}
		\label{asympcovBerry2}
		\lim_{t\rightarrow\infty}{\rm Cov}\left(\frac{\int_{tD}\varphi(B_x) dx}{t^{d/2}w_t^{1/2}}, \frac{\int_{tL}\varphi(B_y) dy}{t^{d/2}w_t^{1/2}}\right)=\frac{1}{2\pi}\int_{S^1}d\theta\int_0^\infty dl\left(-\frac{d}{dl}g_{D,L}(l\theta)\right)l=  \frac{1}{2\pi}\int_D\int_L \frac{dxdy}{|x-y|}.
	\end{equation}
	For the last equality, see Remark \ref{comparison}.
	Note that the integral on the RHS of \eqref{asympcovBerry2} is finite, since $g_{D,L}$ has bounded derivative and compact support if $D,L$ are compact sets with finite perimeter (see Proposition \ref{prop crosscov}).
	
	Let us define the following class of compact sets:
	\[
	\mathcal{D}_O:=\left\{D\subseteq \R^d \text{ compact } : {\rm Vol}(D)>0, |\mathcal{F}[\mathbf{1}_D](x)|=O(|x|^{-d/2})\text{ as $|x|\rightarrow\infty$}\right\},
	\]
	where $\mathcal{F}$ denotes the Fourier transform. For example, $D\in\mathcal{D}_O$ when $D$ is compact, $D=\bar{\mathring{D}}$ and $\partial D$ is smooth with non-vanishing Gaussian curvature, see \cite{Brandolini}.
	By using reduction techniques (see e.g. the proof of \cite[Proposition 4]{MN}) in (\ref{Berry2conv}) and the spectral CLT \cite[Theorem 2]{MN}, we obtain for every $D\in\mathcal{D}_O$ :
	\begin{equation}
		\label{Berry2conv}
		\E\left[\left(\frac{\int_{tD}\varphi(B_x)dx}{\sqrt{{\rm Var}\left(\int_{tD}\varphi(B_x)dx\right)}}-sgn(a_2)\frac{\int_{tD}H_2(B_x)dx}{\sqrt{{\rm Var}\left(\int_{tD}H_2(B_x)dx\right)}}\right)^2\right]\rightarrow0
	\end{equation}
	and 
	\begin{equation}
		\label{Berryconv2bis}
		\frac{\int_{tD}H_2(B_x)dx}{\sqrt{{\rm Var}\left(\int_{tD}H_2(B_x)dx\right)}}\overset{law}{\rightarrow}N(0,1).
	\end{equation}
	
	Moreover, combining (\ref{asympcovBerry2}), (\ref{Berry2conv}), (\ref{Berryconv2bis}) and Theorem \ref{PTtheo} with $\mathcal{D'}=\mathcal{D}\cap \mathcal{D}_O$, we obtain the multi-dimensional central limit theorem
	\begin{equation}
		\label{NICEbis}
		\left(\frac{\int_{tD}\varphi(B_x)dx}{t^{d/2}w_t^{1/2}}\right)_{D\in\mathcal{D}\cap \mathcal{D}_O}\overset{f.d.d.}{\to}(G'(D))_{D\in\mathcal{D}\cap \mathcal{D}_O},
	\end{equation}
	where the set-indexed Gaussian field $G'=(G'(D))_{D\in\mathcal{D}\cap \mathcal{D}_O}$ has covariances given by \eqref{asympcovBerry2}. Note that the restriction on $\mathcal{D}$ in \eqref{NICEbis} is needed to exclude the compact sets with infinite perimeter, since in the latter cases Theorem \ref{isoasymptotics} can not be applied as in \eqref{asympcovBerry2} to verify condition (iii) of Theorem \ref{PTtheo}. Moreover, the restriction on $\mathcal{D}_O$ in \eqref{NICEbis} is needed to conclude \eqref{Berryconv2bis} by means of Theorem \cite[Theorem 2]{MN}, verifying the condition (ii)  in Theorem \ref{PTtheo}.
\end{example}
\begin{example}{\textbf{(Berry with $R=1$.)}}\label{Berry1}
	If $\varphi$ is \textbf{linear}, then $\varphi(x)=a_1 x=a_1 H_1(x)$, where $H_1(x)=x$ is the first Hermite polynomial and $a_1\neq 0$ by definition of Hermite rank $R$.
	In this case, $w_t=a_1^2 w_{1,t}$ is not regularly varying, see Remark \ref{rem:w1}. Therefore, we can not apply Theorem \ref{isoasymptotics} to compute the asymptotic covariances.
	
	Let us now assume that $\varphi(x)$ is \textbf{not linear}, splitting the functional as follows 
	\[
	\int_{tD}\varphi(B_x)dx=a_1\int_{tD}H_1(B_x)dx+\int_{tD}\hat{\varphi}(B_x)dx,
	\]
	where $\hat{\varphi}:=\varphi-a_1H_1\neq0$  because $\varphi$ is not linear (i.e. $\exists\, q\ge2$ with $a_q\neq0$). In addition, recall the notation $\mathcal{D}$ for the class of compact sets with finite perimeter, and consider 
	\[
	\mathcal{D}_o:=\left\{D\subseteq \R^d \text{ compact } : {\rm Vol}(D)>0, |\mathcal{F}[\mathbf{1}_D](x)|=o(|x|^{-d/2})\text{ as $|x|\rightarrow\infty$}\right\}.
	\]
	For example, $D\in\mathcal{D}_o$ when $D$ is compact, $D=\bar{\mathring{D}}$ and $\partial D$ is smooth with non-vanishing Gaussian curvature, see \cite{Brandolini}.
	
	Then, we can study the asymptotic covariances and multi-dimensional fluctuations of $\int_{tD}\varphi(B_x)dx$, as we did in the previous examples,  following two steps:
	\begin{enumerate}
		\item We observe that as $t\rightarrow\infty$ (note that the equality holds by definition of $\hat{\varphi}$)
		\begin{equation}
			\label{Equivalence berry}
			\E\left[\left(\frac{\int_{tD}\varphi(B_x)dx}{t^{d/2}\left(\sum_{q=2}^\infty a_q^2 w_{q,t}\right)^{1/2}}-\frac{\int_{tD}\hat{\varphi}(B_x)dx}{t^{d/2}\left(\sum_{q=2}^\infty a_q^2 w_{q,t}\right)^{1/2}}\right)^2\right]=\frac{{\rm Var}\left(\int_{tD}a_1H_1(B_x)dx \right)}{t^d\sum_{q=2}^\infty a_q^2 w_{q,t}}\rightarrow0, \quad \quad\forall D\in\mathcal{D}_o.
		\end{equation}
		In fact, by Remark \ref{rem:w1} we have (note that $t^{d}=O\left(t^d\sum_{q=2}^\infty a_q^2 w_{q,t}\right)$ by Lemma \ref{ratesberry}, since $\exists\, q\ge2$ with $ a_q\neq0$)
		\[
		{\rm Var}\left(a_1\int_{tD}H_1(B_x) dx\right)=o\left(t^d\sum_{q=2}^\infty a_q^2 w_{q,t}\right), \quad \quad \forall D\in\mathcal{D}_o.
		\]
		\item Since the asymptotic $L^2$-equivalence \eqref{Equivalence berry} holds, the problem of studying the asymptotic covariances and the multi-dimensional Gaussian fluctuations of $\int_{tD}\varphi(B_x)dx$ is reduced to the same problem for $\int_{tD}\hat{\varphi}(B_x)dx$.  Hence, since $\hat{\varphi}$ has Hermite rank greater or equal than $2$, one can solve the problem using the results in the previous examples. For instance, if the Hermite rank of $\hat{\varphi}$ is $2$, then by Example \ref{Berry2} we have
		\[
		\left(\frac{\int_{tD}\hat{\varphi}(B_x)dx}{t^{d/2}\left(\sum_{q=2}^\infty a_q^2 w_{q,t}\right)^{1/2}}\right)_{D\in\mathcal{D}\cap \mathcal{D}_O}\overset{f.d.d.}{\to}(G'(D))_{D\in\mathcal{D}\cap \mathcal{D}_O},
		\]
		and combining the latter with \eqref{Equivalence berry} we obtain (note that $\mathcal{D}_o\subseteq \mathcal{D}_O$, so we add an additional restriction on the class of domains in order to use  \eqref{Equivalence berry}).
		\[
		\left(\frac{\int_{tD}\varphi(B_x)dx}{t^{d/2}\left(\sum_{q=2}^\infty a_q^2 w_{q,t}\right)^{1/2}}\right)_{D\in\mathcal{D}\cap \mathcal{D}_o}\overset{f.d.d.}{\to}(G'(D))_{D\in\mathcal{D}\cap \mathcal{D}_o},
		\]
		where the set-indexed Gaussian field $G'=(G'(D))_{D\in\mathcal{D}\cap \mathcal{D}_o}$ has covariances given by \eqref{asympcovBerry2}. 
	\end{enumerate}
	
\end{example}

\begin{remark}
	Note that all the arguments above could be generalized to prove analogous multi-dimensional central limit theorems for the $d$-dimensional Berry's random wave model (and partially extended to a larger class of fields, see \cite[Section 6]{MN}). We only focused on the case $d=2$ for the sake of brevity. 
\end{remark}
\section{Non-radial covariance functions.} \label{sectionfinal}When $d\ge2$, the covariance function $K$ of the field $A$ could be non-radial, but Theorem \ref{isoasymptotics} can still be applied taking $\mathcal{D}$ as the set of balls centered at a fixed point, see (\ref{case2}).

The goal of this section is showing that the problem \eqref{problem} for non-radial, continuous $K$ can be reduced (for $D,L$ balls centered at a fixed point) to the same problem for a radial covariance function $K_{\text{iso}}$, which will be defined in the sequel. 
This reduction principle will be proved in Proposition \ref{propnonradcov}.\\

In order to state Proposition \ref{propnonradcov}, we need to introduce some additional quantities. 

Let $A=(A_x)_{x\in\R^d}$ be a measurable, weakly stationary random field with \textbf{continuous} covariance function $K:\R^d \rightarrow \R$, $K(0)=1$. By Bochner's theorem, there exists a unique, symmetric, Borel probability measure $G$ on $\R^d$ such that 
\[
K(z)=\int_{\R^d}e^{i\langle x,\lambda \rangle}G(d\lambda).
\]
$G$ is called the \textbf{spectral measure} of $A$ (or associated to $K$).
 In other words, $K$ is the characteristic function of $\Lambda$, where $\Lambda\sim G$ is a random variable with values in $\R^d$. 
 
Note that we can always write
\[
K(z)=\E\left[e^{i\langle z, \Lambda \rangle}\right]=\E\left[e^{i|\Lambda|\langle z, \hat{\Lambda} \rangle}\right],
\]
where $|\Lambda|\sim \mu$  and $\hat{\Lambda}:=\mathbf{1}_{\Lambda\neq0}\,\Lambda/|\Lambda|\sim \sigma$ are respectively the random norm  and the random direction of $\Lambda$. We will refer respectively to the probability measures $\mu$ on $\R_+$ and $\sigma$ on $S^{d-1}$ as the \textbf{isotropic spectral measure} and \textbf{spherical spectral measure} of $A$ (or associated to $K$ or $G$).

\begin{remark}\label{rem rad sphere}
	Let us denote by $\nu$ the uniform probability measure on $S^{d-1}$. Note that $K$ is radial if and only if $\sigma=\nu$.
Indeed, $K$ is radial, i.e. $K(z)=K(Pz)$ for every $P$ orthogonal matrix, if and only if $P\Lambda\sim \Lambda$ for every $P$ orthogonal matrix, that is $\hat{\Lambda}$ is uniformly distributed on $S^{d-1}$ (i.e. $\sigma=\nu$). In particular, given a probability measure $\mu$ on $\R_+$, there is a unique continuous, radial covariance function with isotropic spectral measure $\mu$ (and spherical spectral measure $\nu$).
\end{remark}
Thanks to Remark \ref{rem rad sphere}, we can now give the crucial definition of this section.
\begin{definition}
	Let $A=(A_x)_{x\in\R^d}$ be a weakly stationary random field with continuous covariance function $K:\R^d\rightarrow \R$, isotropic spectral measure $\mu$ and spherical spectral measure $\sigma$, that is
	\[
	K(z)=\int_0^\infty \mu(dr)\int_{S^{d-1}}\sigma(d\theta)\,e^{i\langle r\theta,z \rangle}.
	\] 
	The \textbf{isotropic covariance function} of $A$ is the only continuous, radial covariance function $K_{\text{iso}}:\R^d \rightarrow\R$ with isotropic spectral measure $\mu$, namely (see Remark \ref{rem rad sphere})
	\[
	K_{\text{iso}}(z):=\int_0^\infty \mu(dr)\int_{S^{d-1}}\nu(d\theta)\,e^{i\langle r\theta,z \rangle},
	\]
	where $\nu(d\theta)=d\theta/\omega_{d-1}$ is the uniform probability measure on $S^{d-1}$.
\end{definition}

The main result of this section is the following.

\begin{proposition}\label{propnonradcov}
	Let $A=(A_x)_{x\in\R^d}$ be a measurable, weakly stationary random field with continuous covariance function $K:\R^d\rightarrow \R$ and isotropic covariance function $K_{\text{iso}}$. Then we have (recall the definition \eqref{w} of $w$)
	\begin{equation}\label{equiw}
		w_t^{\text{iso}}:=\int_{\{|z|\le t\}} K_{\text{iso}}(z)dz=\int_{\{|z|\le t\}}K(z)dz=w_t, \quad\quad t>0,
	\end{equation}
	and for $D,L$ balls centered at the same point
	\begin{equation}
		\label{equicov}
		\int_D\int_LK(x-y)dxdy=\int_{D}\int_{L}K_{\text{iso}}(x-y)dxdy.
	\end{equation}
	In particular, if $w_t=w_t^{\text{iso}}$ is regularly varying with index $\alpha\in(-1,d]$ (see \eqref{rv}), by Theorem \ref{isoasymptotics} we have, for $D,L$ balls centered at the same point (see \eqref{case2}), as $t\rightarrow\infty$
	\begin{align}
		\label{covfinalsec}
		{\rm Cov}\left(\frac{\int_{tD}A_x dx}{t^{d/2}w_t^{1/2}}, \frac{\int_{tL}A_x dx}{t^{d/2}w_t^{1/2}}\right)&=\frac{\int_{tD}\int_{tL}K(x-y)dxdy}{t^d w_t}\nonumber\\
		&=\frac{\int_{tD}\int_{tL}K_{\text{iso}}(x-y)dxdy}{t^dw_t^{\text{iso}}}\to \frac{1}{\omega_{d-1}}\int_{S^{d-1}}d\theta\int_0^\infty dl\left(-\frac{d}{dl}g_{D,L}(l\theta)\right)l^{\alpha}.
	\end{align}
\end{proposition}
\begin{proof}
	By Fubini, we have
	\[
	w_t=\int_{\{|z|\le t\}}K(z)dz=\int_{S^{d-1}}\sigma(d\theta)\int_{0}^\infty \mu(dr)\int_{\{|z|\le t\}}e^{i\langle z,r\theta \rangle}=\int_{S^{d-1}}\sigma(d\theta)\int_{0}^\infty \mu(dr)\mathcal{F}[\mathbf{1}_{\{|z|\le t\}}](r\theta),
	\]
	where $\mathcal{F}$ denotes the Fourier transform. Since  $\mathcal{F}[\mathbf{1}_{\{|z|\le t\}}](r\theta)$ is a radial function (i.e. does not depend on $\theta$) and $\sigma,\nu$ are probability measures on $S^{d-1}$, we have
	\[
	\int_{S^{d-1}}\sigma(d\theta)\int_{0}^\infty \mu(dr)\mathcal{F}[\mathbf{1}_{\{|z|\le t\}}](r\theta)=\int_{S^{d-1}}\nu(d\theta)\int_{0}^\infty \mu(dr)\mathcal{F}[\mathbf{1}_{\{|z|\le t\}}](r\theta)=\int_{\{|z|\le t\}}K_{\text{iso}}(z)dz=w_t^{\text{iso}},
	\]
	where the last equality follows again by Fubini. Therefore, \eqref{equiw} is proved. As a consequence of \eqref{equiw} and \eqref{recurrent}, we also obtain \eqref{equicov}. Moreover, \eqref{covfinalsec} follows by \eqref{equiw}, \eqref{equicov} and a direct application of Theorem \ref{isoasymptotics}.
\end{proof}

Proposition \ref{propnonradcov} allows to reduce the problem \eqref{problem} for $K$ non-radial (and $D,L$ balls centered at the same point) to the easiest problem \eqref{problem} for $K_{\text{iso}}$ radial. 

Note that we can also make the argument in the opposite direction: if $K_{\text{iso}}$ is a radial covariance function and we can check \eqref{rv} for $K_{\text{iso}}$, then by Proposition \ref{propnonradcov} we can apply Theorem \ref{isoasymptotics} not only for random fields with covariance function $K_{\text{iso}}$, but for every random field with isotropic covariance function $K_\text{iso}$.

To be more concrete about the possible applications of Proposition \ref{propnonradcov}, we conclude  with the following example.

\begin{example}
	Fix $d\ge2$ and consider the \textbf{non-Gaussian} random field $A=(A_x)_{x\in\R^d}$ given by the rescaled sum of $N$ independent random waves with random phases $\phi_i$, random directions $\theta_i$ and random wavenumbers $k_i$. Namely, we have
\begin{equation}
	\label{superpositions}
	A_x:=\sqrt{\frac{2}{{N}}}\sum_{i=1}^{N}\cos\left(k_i\langle x,\theta_i \rangle+\phi_i \right), 
\end{equation}
where $\phi_1,\theta_1,k_1,\dots,\phi_N,\theta_N,k_N$ are all independent and for all $i=1,\dots,N$ we have: $\phi_i$ is uniformly distributed on $[0,2\pi]$; $\theta_i\sim \sigma$ symmetric probability measure on $S^{d-1}$; $k_i\sim \mu$ probability measure on $(0,\infty)$. 

Note that when $\sigma$ is the uniform distribution on $S^{d-1}$ and $\mu=\delta_k$ for some constant $k>0$, random fields of the form \eqref{superpositions} are called random superposition of independent waves, and have been extensively studied as good local models for wavefunctions (i.e. eigenfunctions of the Laplacian), see e.g. \cite{Berry,NPR,PV,Vidotto}. 

Moreover, note that $A$ is \textbf{not stationary}, since $A_x$ and $A_y$ have not the same distribution in general. For instance, if $N=1$, in general $A_0=\cos(\phi_1)$ has not the same distribution of $A_x=\cos(k_1\langle x,\theta_1 \rangle +\phi_1)$ for every $x\neq0$.

However, $A$ is \textbf{weakly stationary}, implying that Theorem \ref{isoasymptotics} can be applied to compute the asymptotic covariances of functionals of $A$. Indeed, we have
\begin{align*}
	\E[A_x]&=\sqrt{2N}\E\left[\cos\left(k_1\langle x,\theta_1 \rangle+\phi_1 \right)\right]\\
	&=\sqrt{2N}\left(\E\left[\cos\left(k_1\langle x,\theta_1 \rangle \right)\right]\underbrace{\E\left[\cos(\phi_1)\right]}_{=0}-\E\left[\sin\left(k_1\langle x,\theta_1 \rangle \right)\right]\underbrace{\E\left[\sin(\phi_1)\right]}_{=0}\right)=0
\end{align*}
and
\begin{align*}
	{\rm Cov}(A_x,A_y)=\E\left[A_xA_y\right]&=2\E\left[\cos\left(k_1\langle x,\theta_1 \rangle+\phi_1 \right)\cos\left(k_1\langle y,\theta_1 \rangle+\phi_1 \right)\right]\\
	&=2\E\left[\cos\left(k_1\langle x,\theta_1 \rangle \right)\cos\left(k_1\langle y,\theta_1 \rangle \right)\right]\underbrace{\E\left[\cos^2(\phi_1)\right]}_{=1/2}\\
	&+2\E\left[\sin\left(k_1\langle x,\theta_1 \rangle \right)\sin\left(k_1\langle y,\theta_1 \rangle \right)\right]\underbrace{\E\left[\sin^2(\phi_1)\right]}_{=1/2}\\
	&-2\E\left[\cos\left(k_1\langle x,\theta_1 \rangle \right)\sin\left(k_1\langle y,\theta_1 \rangle \right)\right]\underbrace{\E\left[\cos(\phi_1)\sin(\phi_1)\right]}_{=0}\\
	&-2\E\left[\sin\left(k_1\langle x,\theta_1 \rangle \right)\cos\left(k_1\langle y,\theta_1 \rangle \right)\right]\underbrace{\E\left[\sin(\phi_1)\cos(\phi_1)\right]}_{=0}\\
	&=\E\left[\cos\left(k_1\langle x-y,\theta_1 \rangle \right)\right].
\end{align*}
Therefore, denoting by $K$ the covariance function of $A$, by the symmetry of $\theta_1$ we have
\[
K(z)=\E\left[\cos\left(k_1\langle z,\theta_1 \rangle \right)\right]=\E\left[e^{ik_1\langle z,\theta_1 \rangle}\right]=\int_{\R_+ \times S^{d-1}}e^{i \langle z,r\theta \rangle} (\mu\times \sigma) (dr,d\theta),
\]
implying that the isotropic (resp. spherical) spectral measure of $A$ is $\mu$ (resp. $\sigma$). By Proposition \ref{propnonradcov} the integral covariance function $w$ defined in \eqref{w} does not depend on the spherical spectral measure. In other words, whatever is the distribution of the random directions $\theta_1,\dots, \theta_N$ fixed in the definition \eqref{superpositions} of $A$, one can apply Theorem \ref{isoasymptotics}  with the isotropic covariance function of $A$, i.e.
\[
K_{\text{iso}}(z)=\int_{\R_+ \times S^{d-1}}e^{i \langle z,r\theta \rangle} (\mu\times \nu) (dr,d\theta),
\]
where $\nu$ is the uniform probability measure on $S^{d-1}$.
More precisely, whatever is the covariance function $K$, whatever is the spherical spectral measure $\sigma$, if $K_{\text{iso}}$ satisfies
\[
w^{\text{iso}}_t=\int_{\{|z|\le t\}}K_{\text{iso}}(z)dz=\ell(t)t^{\alpha}, \quad \quad \alpha\in(-1,d],
\]
with $\ell:\R_+\rightarrow \R$ slowly varying, we have for $D,L$ balls centered at the same point
(see \eqref{case2}), as $t\rightarrow\infty$
\begin{equation*}
	{\rm Cov}\left(\frac{\int_{tD}A_x dx}{t^{d/2}(w^{\text{iso}}_t)^{1/2}}, \frac{\int_{tL}A_x dx}{t^{d/2}(w^{\text{iso}}_t)^{1/2}}\right)\to \frac{1}{\omega_{d-1}}\int_{S^{d-1}}d\theta\int_0^\infty dl\left(-\frac{d}{dl}g_{D,L}(l\theta)\right)l^{\alpha}.
\end{equation*}
\end{example}
\section*{Acknowledgements}
I would like to thank the reviewers for the constructive remarks and useful suggestions.
I would like to thank my supervisor, Prof. Ivan Nourdin, for the patient guidance and advice provided during the writing of this article.
I acknowledge the support of the Luxembourg National Research Fund PRIDE17/1224660/GPS. 
\bibliographystyle{plain}

\end{document}